\documentclass[preprint,12pt]{elsarticle}

\usepackage[utf8]{inputenc}
\usepackage[T1]{fontenc}
\usepackage{lmodern}
\usepackage{ifpdf}
\usepackage[textwidth=390pt, top=2.4cm, bottom=2.8cm]{geometry}
\usepackage{booktabs}
\usepackage[font=footnotesize]{caption}
\usepackage{changepage}
\usepackage[shortlabels]{enumitem}
\usepackage[setbuildercolon]{matmatix}
\usepackage[linktoc=all]{hyperref}

\ifpdf
  \hypersetup{ 
    pdfauthor = {François Bienvenu}, 
    pdftitle = {The split-and-drift random graph, a null model for speciation}
  }
\fi

\hypersetup{colorlinks, allcolors = {MidnightBlue}}

\makeatletter
\newtheorem*{rep@theorem}{\rep@title}
\newcommand{\newreptheorem}[2]{%
\newenvironment{rep#1}[1]{%
 \def\rep@title{#2 \ref*{##1}}%
 \begin{rep@theorem}}%
 {\end{rep@theorem}}}
\makeatother

\newtheorem{theorem}{Theorem}[section]
\newreptheorem{theorem}{Theorem}
\newtheorem{proposition}[theorem]{Proposition}
\newreptheorem{proposition}{Proposition}
\newtheorem{corollary}[theorem]{Corollary}
\newreptheorem{corollary}{Corollary}
\newtheorem*{theorem*}{Theorem}
\newtheorem{quotedtheorem}{Theorem}

\newtheorem{conjecture}[theorem]{Conjecture}
\newreptheorem{conjecture}{Conjecture}
\newtheorem{lemma}[theorem]{Lemma}
\newreptheorem{lemma}{Lemma}
\theoremstyle{remark}
\newtheorem{remark}[theorem]{Remark}
\theoremstyle{definition}
\newtheorem{definition}[theorem]{Definition}
\newtheorem{quoteddefinition}{Definition}

\theoremstyle{empty}
\newenvironment{mathlist}
{\begin{enumerate}[label={\upshape(\roman*)}, align=left, widest=iii, leftmargin=*]}
{\end{enumerate}\ignorespacesafterend}

\newcommand{\SmallOh}[1]{o\mleft(#1\mright)}

\newcommand{\IsSmallOh}[2]{#1 = o\mleft(#2\mright)}

\newcommand{\SmallOhP}[1]{o_p\mleft(#1\mright)}

\newcommand{\IsSmallOhP}[2]{#1 = o_p\mleft(#2\mright)}

\newcommand{\IsBigOh}[2]{#1 = O\mleft(#2\mright)}

\newcommand{\IsBigOhP}[2]{#1 = O_p\mleft(#2\mright)}

\newcommand{\IsBigTheta}[2]{#1 = \Theta\mleft(#2\mright)}

\newcommand{\SmallOmega}[1]{\omega\mleft(#1\mright)}

\newcommand{\IsSmallOmega}[2]{#1 = \omega\mleft(#2\mright)}

\newcommand{\aas}{\text{ a.a.s. }}
\newcommand{\marked}[1]{#1^\bullet}
\newcommand{\unmarked}[1]{#1^\circ}
\newcommand{\lkto}{\leftrightarrow}

\newcommand{\nbcc}{\#\mathrm{CC}}

\begin{document}

\begin{frontmatter}

\title{The split-and-drift random graph, a null model for speciation}

\journal{Stochastic Processes and their Applications}

\author[cirb,lpma]{François~Bienvenu\corref{cor}}
\author[lpma]{Florence~Débarre}
\author[cirb,lpma]{Amaury~Lambert}

\address[cirb]{Center for Interdisciplinary Research in Biology (CIRB),
CNRS UMR 7241, Collège de France, PSL Research University, 
Paris, France}

\address[lpma]{Laboratoire de Probabilités, Statistique et Modélisation (LPSM),
CNRS UMR 8001, Sorbonne Université,
Paris, France}

\cortext[cor]{Corresponding author: francois.bienvenu@normalesup.org}

\begin{abstract}
We introduce a new random graph model motivated by biological questions
relating to speciation. This random graph is defined as the stationary
distribution of a Markov chain on the space of graphs on $\{1, \ldots, n\}$.
The dynamics of this Markov chain is governed by two types of events: vertex
duplication, where at constant rate a pair of vertices is sampled uniformly and
one of these vertices loses its incident edges and is rewired to the other
vertex and its neighbors; and edge removal, where each edge disappears at
constant rate.  Besides the number of vertices $n$, the model has a single
parameter $r_n$.

Using a coalescent approach, we obtain explicit formulas for the first moments
of several graph invariants such as the number of edges or the number of
complete subgraphs of order $k$. These are then used to identify five
non-trivial regimes depending on the asymptotics of the parameter $r_n$.
We derive an explicit expression for the degree distribution, and show
that under appropriate rescaling it converges to classical distributions when
the number of vertices goes to infinity. Finally, we give asymptotic bounds for
the number of connected components, and show that in the sparse regime the
number of edges is Poissonian.
\end{abstract}


\end{frontmatter}

\tableofcontents

\section{Introduction}
%[[[

%[[[
In this paper, we introduce a random graph derived from a minimalistic model of
speciation. This random graph bears superficial resemblance to classic models
of protein interaction networks \cite{ChungJCB2003, IspolatovPhysRevE2005,
SoleACS2002, VazquezComPlexUs2003} in that the events shaping the graph are
the duplication of vertices and the loss of edges. However, our model is
obtained as the steady state of a Markov process (rather than by repeatedly
adding vertices), and has the crucial feature that the duplication of vertices
is independent from the loss of edges. These differences result in a very
different behavior of the model.

Before describing the model formally in Section~\ref{secFormalDescription},
let us briefly explain the motivation behind its introduction.
%]]]

\subsection{Biological context} \label{secBiologicalMotivation}
%[[[
Although it is often presented as central to biology, there is no consensus
about how the concept of species should be defined. A widely held view is
that it should be based on the capacity of individuals to interbreed.
This is the so-called ``biological species concept'', wherein a species
is defined as a group of potentially interbreeding populations that cannot
interbreed with populations outside the group.

This view, whose origins can be traced back to the beginning of the 20th
century \cite{Poulton1904}, was most famously promoted by Ernst Mayr
\cite{Mayr1942} and has been most influential in biology 
\cite{CoyneOrr2004}.  However, it remains quite imprecise: indeed, groups of
populations such that (1) all pairs of populations can interbreed and (2) no
population can interbreed with a population outside the group are probably not
common in nature -- and, at any rate, do not correspond to what is considered
a species in practice.  Therefore, some leniency is required when applying
conditions (1) and (2). But once we allow for this, there are several ways to
formalize the biological species concept, as illustrated in
Figure~\ref{figConflictingDefinitions}.  Thus, it seems arbitrary to favor one
over the others in the absence of a mechanism to explain why some kind of
groups should be more relevant (e.g., arise more frequently) than others.

\begin{figure}[h]
  \centering
  \captionsetup{width=0.85\linewidth}
  \includegraphics[width=0.85\linewidth]{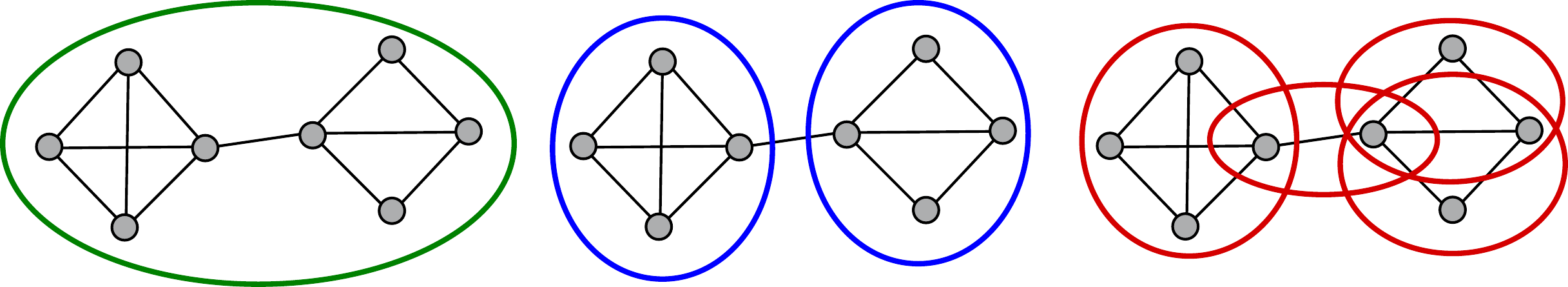}
  \caption[Conflicting definitions]
  {The vertices of the graph represent populations and its edges denote
  interbreeding potential (that is, individuals from two linked populations
  could interbreed, if given the chance). Even with such perfect information,
  it is not obvious how to delineate ``groups of potentially interbreeding
  populations that cannot interbreed with populations outside the group'':
  should these correspond to connected components (on the left, in green),
  maximal complete subgraphs (on the right, in red), or be based on some
  other clustering method (middle, in blue)?}
  \label{figConflictingDefinitions}
\end{figure}

Our aim is to build a minimal model of speciation that would make
predictions about the structure and dynamics of the interbreeding network and
allow one to recover species as an emergent property. To do so, we model
speciation at the level of populations. Thus, we consider a set of $n$
populations and we track the interbreeding ability for every pair of
populations.  All this information is encoded in a graph whose vertices
correspond to populations and whose edges indicate potential interbreeding,
i.e., two vertices are linked if and only if the corresponding populations can
interbreed.

Speciation will result from the interplay between two mechanisms.
First, populations can sometimes ``split'' into two initially identical
populations which then behave as independent entities; this could happen
as a result of the fragmentation of the habitat or of the colonization of a new
patch. Second, because they behave as independent units, two initially
identical populations will diverge (e.g., as a result of genetic drift)
until they can no longer interbreed.
%]]]

\subsection{Formal description of the model} \label{secFormalDescription}
%[[[
Start from any graph with vertex set $V = \Set{1, \ldots, n}$, and
let it evolve according to the following rules
\begin{enumerate}
  \item \textbf{Vertex duplication:} each vertex ``duplicates'' at rate~$1$;
    when a vertex duplicates, it chooses another vertex uniformly at random
    among the other vertices and replaces it with a copy of itself.
    The replacement of $j$ by a copy of $i$ means that $j$ loses its
    incident edges and is then linked to $i$ and to all of its neighbors, as
    depicted in Figure~\ref{figVertexDuplication}.
    \begin{figure}[h]
      \centering
      \captionsetup{width=0.8\linewidth}
      \includegraphics[width=0.8\linewidth]{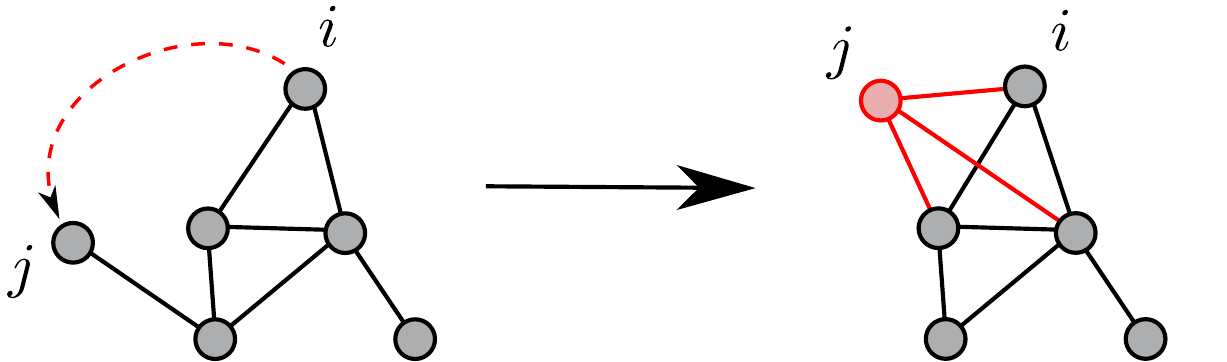}
      \caption[Vertex duplication]
       {An illustration of vertex duplication. Here, $i$ duplicates and
        replaces $j$. After the duplication, $j$ is linked to $i$
        and to each of its neighbors.}
      \label{figVertexDuplication}
    \end{figure}
  \item \textbf{Edge removal:} each edge disappears at constant rate~$\rho$.
\end{enumerate}

This procedure defines a continuous-time Markov chain $(G_n(t))_{t \geq 0}$
on the finite state space of all graphs whose vertices are the integers
$1, \ldots, n$.  It is easy to see that this Markov chain is irreducible.
Indeed, to go from any graph $G^{(1)}$ to any graph $G^{(2)}$, one can consider
the following sequence of events: first, a vertex is duplicated repeatedly in
order to obtain the complete graph of order~$n$ (e.g., $\forall k \in \{2,
\ldots, n\}$, vertex~$k$ is replaced by a copy of vertex~$1$); then, all the
edges that are not in $G^{(2)}$ are removed.

Because the Markov chain $(G_n(t))_{t \geq 0}$ is irreducible, it has a
unique stationary probability distribution $\mu_{n, \rho}$. This probability
distribution on the set of graphs of order~$n$ defines a random graph that is
the object of study of this paper.
%]]]

\subsection{Notation}
%[[[
To study the asymptotic behavior of our model as $n \to +\infty$, we can let
$\rho$, the ratio of the edge removal rate to the vertex duplication rate, be
a function of $n$. As will become evident, it is more convenient
to parametrize the model by
\[
  r_n \defas \tfrac{n - 1}{2} \, \rho_n \, .
\]
Thus, we write $G_{n, r_n}$ to refer to a random graph whose law is
$\mu_{n, {\frac{2r_n}{n - 1}}}$.

Although some of our results hold for any $(n, r)$, in many cases we will
be interested in asymptotic properties that are going to depend on the
asymptotics of $r_n$. To quantify these, we will use the Bachmann--Landau
notation, which for positive sequences $r_n$ and $f(n)$ can be summarized as:
\begin{itemize}
  \item $r_n \sim f(n)$ when $r_n / f(n) \to 1$.
  \item $r_n = o(f(n))$ when $r_n / f(n) \to 0$.
  \item $r_n = \Theta(f(n))$ when there exists positive constants $\alpha$ and
    $\beta$ such that, asymptotically, $\alpha f(n) \leq r_n \leq \beta f(n)$.
  \item $r_n = \omega(f(n))$ when $r_n / f(n) \to +\infty$.
\end{itemize}

These notations also have stochastic counterparts, whose meaning will be
recalled when we use them.

Finally, we occasionally use the expression \emph{asymptotically almost
surely} (abbreviated as a.a.s.) to that a property holds with probability that
goes to $1$ as $n$ tends to infinity:
\[
  \mathcal{Q}_n\; \aas \iff\; \Prob{\mathcal{Q}_n} \tendsto{n \to +\infty} 1 \,.
\]
%]]]

\subsection{Statement of results}
%[[[
Table~\ref{tabFirstMoments} lists the first moments of several graph invariants
obtained in Section~\ref{secFirstMoments}. These are then used to 
identify different regimes, depending on the asymptotic behavior of the
parameter $r_n$, as stated in Theorem~\ref{thmRegimes}.

\begin{reptheorem}{thmRegimes}
Let $D_n$ be the degree of a fixed vertex of $G_{n, r_n}$. In the limit as
$n \to +\infty$, depending on the asymptotics of $r_n$ we have the following
behaviors for~$G_{n, r_n}$
\begin{mathlist}
\item \emph{Complete graph:} when $\IsSmallOh{r_n}{1/n}$, 
  $\Prob{G_{n, r_n} \text{ is complete}}$ goes to~$1$,
  while when $\IsSmallOmega{r_n}{1/n}$ it goes to~$0$;
  when $\IsBigTheta{r_n}{1/n}$, this probability is bounded away from~$0$
  and from~$1$.
\item \emph{Dense regime:} when $\IsSmallOh{r_n}{1}$, $\Prob{D_n = n - 1} \to 1$.
\item \emph{Sparse regime:} when $\IsSmallOmega{r_n}{n}$,
  $\Prob{D_n = 0} \to 1$.
\item \emph{Empty graph:} when $\IsSmallOh{r_n}{n^2}$, 
  $\Prob{G_{n, r_n} \text{ is empty}}$ goes to~$0$
  while when $\IsSmallOmega{r_n}{n^2}$ it goes to~$1$;
  when $\IsBigTheta{r_n}{n^2}$, this probability is bounded away from~$0$ and
  from~$1$.
\end{mathlist}
\end{reptheorem}

% Ideally the table should be placed *before* the theorem
% But this was not possible here
\begin{table}[h!]
\begin{adjustwidth}{-1cm}{-1cm}
\centering
\begin{tabular}{llll}
\toprule
 Variable   &  Expectation  &  Variance  &  Covariance \\
\midrule
$\Indic{i \lkto j}$  &  $\frac{1}{1 + r}$  &  $\frac{r}{(1 + r)^2}$  &
  $\frac{r}{(1 + r)^2 (3 + 2 r)}$ \small if vertex in common,\\
  & & & $\frac{2\,r}{(1 + r)^2 (3 + r) (3 + 2r)}$ \small otherwise.\\ \addlinespace
$D^{(i)}_n$  &  $\frac{n - 1}{1 + r}$  &
  $\frac{r (n - 1) (1 + 2 r + n)}{(1 + r)^2 \, (3 + 2 r)} $  &
  $\frac{r}{(1 + r)^2} \mleft(1 + \frac{3 (n - 2)}{3 + 2 r} +
     \frac{2 (n - 2)(n - 3)}{(3 + r)(3 + 2 r)}\mright)$ \\ \addlinespace
$|E_n|$  &  $\frac{n (n - 1)}{2(1 + r)}$  &
  $\frac{r n (n - 1) (n^2 + 2 r^2 + 2nr + n + 5r + 3)}{2 \, (1 + r)^2 
    (3 + r) \, (3 + 2r)}$  &  \,---  \\ \addlinespace
$X_{n, k}$  &  $\binom{n}{k} \mleft(\frac{1}{1 + r}\mright)^{k - 1}$  & 
  \,{\footnotesize unknown} & \,--- \\
\bottomrule
\end{tabular}
\end{adjustwidth}
\caption{First and second moments of several graph invariants of $G_{n, r}$:
$\Indic{i \lkto j}$ is the variable indicating that $\{ij\}$ is an edge,
$D^{\scriptscriptstyle (i)}_n$ the degree of vertex~$i$, $\Abs{E_n}$ the number
of edges and $X_{n, k}$ the number of complete subgraphs of order~$k$. The
covariance of the indicator variables of two edges depends on whether these
edges share a common end, hence the two expressions. All expressions hold for
every value of $n$ and $r$.}
\label{tabFirstMoments}
\end{table}

In Section~\ref{secDegreeDistribution}, we derive an explicit expression
for the degree distribution, which holds for every value of $n$ and $r_n$. We
then show that, under appropriate rescaling, this degree converges to
classical distributions.

\begin{reptheorem}{thmDegreeDistribution}
Let $D_n$ be the degree of a fixed vertex of $G_{n, r_n}$. Then,  
  for each $k \in \Set{0, \ldots, n - 1}$,
\[
  \Prob{D_n = k} = \frac{2\,r_n\,(2\,r_n + 1)}{(n + 2\,r_n) (n - 1 + 2\,r_n)} \,
  (k + 1) \, \prod_{i = 1}^k \frac{n - i}{n - i + 2\,r_n - 1}\, , 
\]
where the empty product is 1.
\end{reptheorem}

\begin{reptheorem}{thmDegreeConvergence} ~
\begin{mathlist}
  \item If $r_n \to r > 0$, then $\frac{D_n}{n}$ converges in distribution to a 
    $\mathrm{Beta}(2, 2\,r)$ random variable.
  \item If $r_n$ is both $\SmallOmega{1}$ and $\SmallOh{n}$, then
    $\frac{D_n}{n / r_n}$ converges in distribution to a size-biased
    exponential variable with parameter~$2$.
  \item If $2\,r_n / n \to \rho > 0$, then $D_n + 1$ converges in distribution
    to a size-biased geometric variable with parameter $\rho/(1 + \rho)$.
\end{mathlist}
\end{reptheorem}

\newpage
\newgeometry{textwidth=390pt, top=2.0cm, bottom=2.6cm}
\enlargethispage{0.6cm}

Asymptotic bounds for the number of connected components are obtained in
Section~\ref{secConnectedComponents}, where the following theorem is proved.

\begin{reptheorem}{thmConnectedComponents}
Let $\nbcc_n$ be the number of connected components of $G_{n, r_n}$. If
$r_n$ is both $\SmallOmega{1}$ and $\SmallOh{n}$, then
\[
  \frac{r_n}{2} + \SmallOhP{r_n}
  \leq \; \nbcc_n \; \leq
  2\,r_n \log n + \SmallOhP{r_n \log n}
\]
where, for a positive sequence $(u_n)$, $\SmallOhP{u_n}$ denotes a given
sequence of random variables $(X_n)$ such that $X_n/u_n \to 0$ in probability.
\end{reptheorem}

Because the method used to obtain the upper bound in
Theorem~\ref{thmConnectedComponents} is rather crude, we formulate the
following conjecture, which is well supported by simulations.

\begin{repconjecture}{conjConnectedComponents}
\[
  \exists \alpha, \beta > 0 \st
  \Prob{\alpha r_n \leq \nbcc_n \leq \beta r_n} \tendsto{n \to \infty} 1.
\]
\end{repconjecture}

Finally, in Section~\ref{secNumberOfEdges} we use the Stein--Chen method to
show that the number of edges is Poissonian in the sparse regime, as shown
by Theorem~\ref{thmPoissonianEdges}.

\begin{reptheorem}{thmPoissonianEdges}
Let $|E_n|$ be the number of edges of $G_{n, r_n}$. If $\IsSmallOmega{r_n}{n}$
then
\[
  d_{\mathrm{TV}}\big(|E_n|, \mathrm{Poisson}(\lambda_n)\big)
  \tendsto{n \to +\infty} 0 \, , 
\]
where $d_{\mathrm{TV}}$ stands for the total variation distance and
$\lambda_n = \Expec{|E_n|} \sim \frac{n^2}{2r_n}$. If in addition
$\IsSmallOh{r_n}{n^2}$, then $\lambda_n \to +\infty$ and as a result
\[
 \frac{|E_n| - \lambda_n}{\sqrt{\lambda_n}}
 \tendsto[\mathcal{D}]{n\to+\infty} \mathcal{N}(0, 1) \, ,
\]
where $\mathcal{N}(0, 1)$ denotes the standard normal distribution.
\end{reptheorem}

These results are summarized in Figure~\ref{figMainResults}.

\begin{figure}[h]
  \centering
  \captionsetup{width=1.2\linewidth}
\begin{adjustwidth}{-2cm}{-2cm}
  \includegraphics[width=\linewidth]{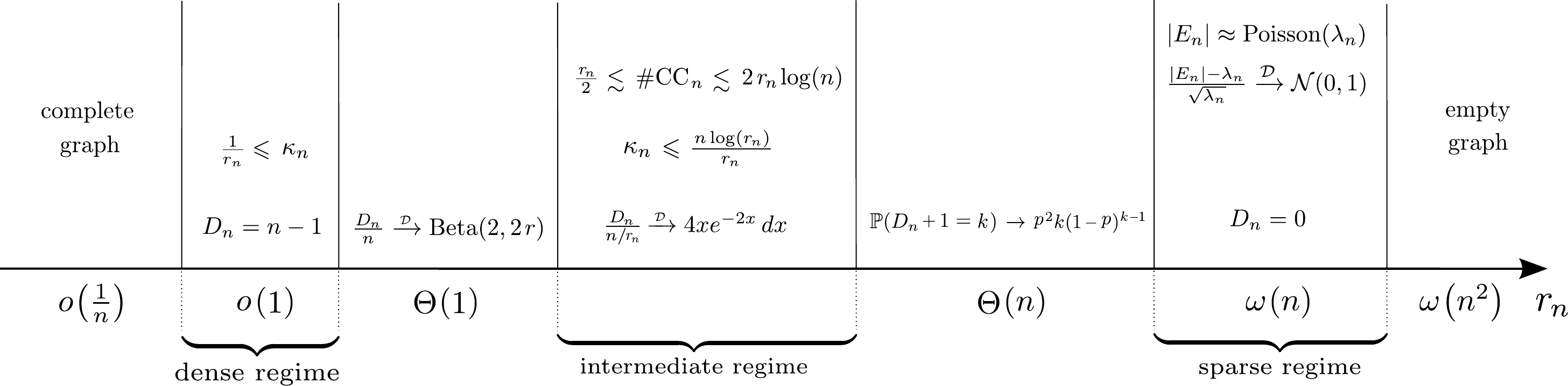}
\end{adjustwidth}
  \caption[Main results]
  {A graphical summary of the main results established in the paper; $D_n$ is
  the degree of a fixed vertex, $|E_n|$ the number of edges, $\nbcc_n$ the
  number of connected components, and $\kappa_n$ the clique number. All
  equalities and inequalities are to be understood ``asymptotically almost
  surely'' (i.e.\ hold with probability that goes to $1$ as $n$ tends to
  infinity).}
  \label{figMainResults}
\end{figure}
%]]]

%]]]

\newpage
\newgeometry{textwidth=390pt, top=2.4cm, bottom=2.8cm}

\section{Coalescent constructions of $G_{n, r_n}$} \label{secCoalescentConstructions}
%[[[

In this section, we detail coalescent constructions of $G_{n, r_n}$ that will
be used throughout the rest of the paper. Let us start by recalling some
results about the Moran model.

\subsection{The standard Moran process} \label{secMoranProcess}
%[[[
The Moran model \cite{Moran1958} is a classic model of population genetics.
It consists in a set of $n$ particles governed by the following dynamics: after
an exponential waiting time with parameter $\binom{n}{2}$, a pair of particles
is sampled uniformly at random.  One of these particles is then removed (death)
and replaced by a copy of the other (birth), and we iterate the procedure. 

In this document, we will use the Poissonian representation of the Moran
process detailed in the next definition.

\begin{definition} \label{defStandardMoran}
The \emph{driving measure of a standard Moran process on $V$} is a collection
$\mathpzc{M} = (M_{(ij)})_{(ij) \in V^2}$ of i.i.d.\
Poisson point processes with rate~$1/2$ on $\R$. 

We think of the elements of $V$ as sites, each occupied by a
single particle. In forward time, each atom $t \in M_{(ij)}$ indicates the
replacement, at time~$t$, of the particle in~$i$ by a copy of the particle
in~$j$.

For any given time~$\alpha \in \R$, $\mathpzc{M}$ defines a \emph{genealogy
of $V$ on $\OCInterval{-\infty, \alpha}$}. Taking $\alpha = 0$ and working in
backward time, i.e.\ writing $t \geq 0$ to refer to the absolute time~$-t$,
this genealogy is described by a collection of \emph{ancestor functions}
$a_{t}$, $t \in \COInterval{0, +\infty}$, $a_t\colon V \to V$,
defined as follows: $(a_{t})_{t \geq 0}$ is the piecewise constant
process such that
\begin{mathlist}
\item $a_0$ is the identity on $V$.
\item If $t \in M_{(ij)}$ then 
   \begin{itemize}
     \item For all $k$ such that $a_{t-}(k) = i$, $a_t(k) = j$.
     \item For all $k$ such that $a_{t-}(k) \neq i$, $a_t(k) = a_{t-}(k)$.
   \end{itemize}
\item If for all $(ij) \in V^2$,
  $M_{{(ij)}} \cap \ClosedInterval{s, t} = \emptyset$, then
  $a_t = a_s$.
\end{mathlist}

We refer to $a_t(i)$ as the ancestor of $i$ at time~$t$ before the
present -- or, more simply, as the \emph{ancestor of $i$ at time~$t$}.
\end{definition}

The standard Moran process is closely related to the Kingman
coalescent~\cite{Kingman1982}. Indeed, let $\mathcal{R}_t$ denote the
equivalence relation on $V$ defined by
\[
  i \, \mathcal{R}_t \, j \iff a_t(i) = a_t(j) \, , 
\]
and let $K_t = V / \mathcal{R}_t$ be the partition of $V$ induced by
$\mathcal{R}_t$. Then, $(K_t)_{t \geq 0}$ is a Kingman coalescent on~$V$.
In particular, we will frequently use the next lemma.

\begin{lemma} \label{propMoran}
Let $(a_t)_{t \geq 0}$ be the ancestor functions of a standard Moran process
on~$V$. For any $i \neq j$, let
\[
  T_{\{ij\}} = \inf\Set{t \geq 0 \suchthat a_t(i) = a_t(j)}
\]
be the coalescence time of $i$ and $j$ and, for any $S \subset V$, let
\[
  T_S = \inf\Set*{T_{\{ij\}} \suchthat i, j \in S, i \neq j} \, .
\]
Then, for all $t \geq 0$, conditional on $\Set{T_S > t}$, $(T_S - t)$ is an
exponential variable with parameter~$\binom{|S|}{2}$.
\end{lemma}

For a more general introduction to Kingman's coalescent and Moran's model, one
can refer to e.g.\ \cite{Durrett2008} or \cite{Etheridge2011}.
%]]]

\subsection{Backward construction} \label{secBackwardConstruction}
%[[[

We now turn to the description of the coalescent framework on which our
study relies. The crucial observation is that, for $t$ large enough,
every edge of $G_n(t)$ can ultimately be traced back to an initial edge that
was inserted between a duplicating vertex and its copy. To find out whether
two vertices $i$ and~$j$ are linked in $G_n(t)$ , we can trace back the ancestry
of the potential link between them and see whether the corresponding initial
edge and its subsequent copies survived up to time~$t$. The first part of this
procedure depends only on the vertex duplication process and, conditional on
the sequence of ancestors of $\{ij\}$, the second one depends only on the edge
removal process, making the whole procedure tractable. The next proposition
formalizes these ideas.

\begin{proposition} \label{propBackwardRepresentation}
Let $V = \Set{1, \ldots, n}$ and let $V^{(2)}$ be the set of unordered pairs of
elements of~$V$. Let $\mathpzc{M}$ be the driving measure of a standard Moran
process on $V$, and $(a_t)_{t \geq 0}$ the associated ancestor functions
(that is, for each $i$ in $V$, $a_t(i)$ is the ancestor of~$i$ at time~$t$).
Let $\mathpzc{P} = (P_{\{ij\}})_{\{ij\} \in
V^{(2)}}$ be a collection of i.i.d.\ Poisson point processes with rate~$r_n$
on $\COInterval{0, +\infty}$ such that $\mathpzc{M}$ and $\mathpzc{P}$ are
independent.  For every pair $\{ij\}\in V^{(2)}$, define
\[
  P^\star_{\{ij\}} = 
  \Set{t \geq 0 \suchthat t \in P_{\{a_t(i) a_t(j)\}}} \, , 
\]
with the convention that, $\forall k \in V$, $P_{\{k\}} = \emptyset$. Finally,
let $G = \mleft(V, E\mright)$ be the graph defined by
\[
  E = \Set{\{ij\} \in V^{(2)} \suchthat P^\star_{\{ij\}} = \emptyset} \, .
\]
Then, $G \sim G_{n, r_n}$.
\end{proposition}

Throughout the rest of this document, we will write $G_{n, r_n}$ for the graph
obtained by the procedure of Proposition~\ref{propBackwardRepresentation}. 

\begin{proof}[Proof of Proposition~\ref{propBackwardRepresentation}]
First, consider the two-sided extension of
$(G_n(t))_{t \geq 0}$, i.e.\ the corresponding stationary process on $\R$
(see, e.g., Section~{7.1} of \cite{Durrett2010}), which by a slight abuse of
notation we note $(G_n(t))_{t \in \R}$. Next, let $(\bar{G}_n(t))_{t \in \R}$
be the time-rescaled process defined by
\[
  \bar{G}_n(t) = G_n(t(n -1)/2)\,.
\]
This time-rescaled process has the same stationary distribution as
$(G_n(t))_{t \in \R}$ and so, in particular, $\bar{G}_n(0) \sim G_{n, r_n}$.

In the time-rescaled process, each vertex duplicates at rate $(n - 1) / 2$ and
each edge disappears at rate $r_n = (n - 1) \rho_n / 2$. All these events
being independent, we see that the vertex duplications correspond to the
atoms of a standard Moran process on $V = \{1, \ldots, n\}$, and the edge
removals to the atoms of $\binom{n}{2}$ i.i.d.\ Poisson point processes with
rate~$r_n$ on $\R$, that are also independent of the Moran process. Thus,
there exists $(\bar{\mathpzc{M}}, \bar{\mathpzc{P}})$ with the same
law as $(\mathpzc{M}, \mathpzc{P})$ from the proposition and such that,
for $t \geq 0$,
\begin{itemize}
  \item If $t \in \bar{M}_{(ij)}$, then $j$ duplicates and replaces $i$
    in $\bar{G}_n(-t)$.
  \item If $t \in \bar{P}_{\{ij\}}$, then if there is an edge between $i$ and
    $j$ in $\bar{G}_n(-t)$, it is removed.
\end{itemize}
Since $(\bar{\mathpzc{M}}, \bar{\mathpzc{P}})$ has the same law as
$(\mathpzc{M}, \mathpzc{P})$, if we show that
\[
  \{ij\} \in \bar{G}_n(0) \iff \bar{P}^\star_{\{ij\}} = \emptyset \,, 
\]
where $\bar{P}^\star_{\{ij\}} = \Set{t \geq 0 \suchthat t \in
\bar{P}_{\{\bar{a}_t(i) \bar{a}_t(j)\}}}$ is the same deterministic function
of $(\bar{\mathpzc{M}}, \bar{\mathpzc{P}})$ as $P^\star_{\{ij\}}$ of
$(\mathpzc{M}, \mathpzc{P})$,
then we will have proved that $\bar{G}_n(0)$ has the same law as the graph
$G$ from the proposition.

Now to see why the edges of $\bar{G}_n(0)$ are exactly the pairs $\{ij\}$ such
that $\bar{P}_{\{ij\}}$ is empty, note that, in the absence of edge-removal
events, $\bar{G}_n(0)$ is the complete graph and the ancestor
the edge $\{ij\}$ at time~$t$ is $\{a_t(i)\, a_t(j)\}$. Conversely, deleting
the edge $\{k\ell\}$ from $\bar{G}_n(-t)$ will remove all of its subsequent
copies from $\bar{G}_n(0)$, i.e.\ all edges $\{ij\}$ such that
$\{a_t(i)\, a_t(j)\} = \{k\ell\}$. Thus, the edges of $\bar{G}_n(0)$ are
exactly the edges that have no edge-removal events on their ancestral lineage
-- i.e, such that $\bar{P}^\star_{\{ij\}} = \emptyset$.
\end{proof}

Proposition~\ref{propBackwardRepresentation} shows that $G_{n, r_n}$ can
be obtained as a deterministic function of the genealogy $(a_t)_{t \geq 0}$
of a Moran process and of independent Poisson point processes. Our next
result shows that, in this construction, $(a_t)_{t \geq 0}$ can be replaced
by a more coarse-grained process -- namely, a Kingman coalescent (note that the
Kingman coalescent contains less information because it only keeps track of
blocks, not of which ancestor corresponds to which block at a given time~$t$).
This will be useful to give a forward construction of $G_{n, r_n}$ in
Section~\ref{secForwardConstruction}.  The proof of this result is
straightforward and can be found in Section~\ref{appBackwardForward} of the
Appendix.

\begin{proposition} \label{propBackwardWithKingman}
Let $(K_t)_{t \geq 0}$ be a Kingman coalescent on $V = \Set{1, \ldots, n}$,
and let $\pi_t(i)$ denote the block containing~$i$ in the corresponding
partition at time~$t$. Let the associated genealogy of pairs be the set 
\[
  \mathcal{G} = \Set*[\Big]{\big(t,\, \{\pi_t(i)\, \pi_t(j)\}\big) \suchthat
    \{ij\} \in V^{(2)},\, t \in \COInterval{0, T_{\{ij\}}}} \,, 
\]
where $T_{\{ij\}} = \inf\Set{t \geq 0 \suchthat \pi_t(i) = \pi_t(j)}$. 
Denote by
\[
  L_{\{ij\}} = \Set*[\Big]{\big(t,\, \{\pi_t(i)\, \pi_t(j)\}\big) \suchthat
   t \in \COInterval{0, T_{\{ij\}}}}
\]
the lineage of $\{ij\}$ in this genealogy. Finally, let $P^\bullet$ be a
Poisson point process with constant intensity $r_n$ on $\mathcal{G}$ and let
$G = (V, E)$, where
\[
  E = \Set{\{ij\} \in V^{(2)} \suchthat P^\bullet \cap L_{\{ij\}} = \emptyset} \, .
\]
Then, $G \sim G_{n, r_n}$.
\end{proposition}

We finish this section with a technical lemma that will be useful in the
calculations of Section~\ref{secFirstMoments}. Again, the proof of this result
has no interest in itself and can be found in Section~\ref{appBackwardForward}
of the Appendix.

\begin{lemma} \label{lemmaBackwardConstruction}
Let $S$ be a subset of $V^{(2)}$. Conditional on the measure $\mathpzc{M}$, for
any interval $I \subset \COInterval{0, +\infty}$ such that
\begin{mathlist}
\item For all $\{ij\} \in S$, $\forall t \in I$, $a_t(i) \neq a_t(j)$.
\item For all $\{k\ell\} \neq \{ij\}$ in $S$, $\forall t \in I, \;
  \{a_t(i)\, a_t(j)\} \neq \{a_t(k)\, a_t(\ell)\}$,
\end{mathlist}
$P^\star_{\{ij\}} \cap I$, $\{ij\} \in S$, are independent Poisson point
processes with rate~$r_n$ on~$I$.

Moreover, for any disjoint intervals $I$ and $J$,
$(P^\star_{\{ij\}} \cap I)_{\{ij\} \in S}$ is independent of
$(P^\star_{\{ij\}} \cap J)_{\{ij\} \in S}$.
\end{lemma}

Before closing this section, let us sum up our results in words:
if we think of $\{a_t(i)\, a_t(j)\}$ as being the ancestor of $\{ij\}$ at
time~$t$, then the genealogy of vertices induces a genealogy of pairs of
vertices, as illustrated by Figure~\ref{figGenealogyPairs}.
Edge-removal events occur at constant rate~$r_n$ along the branches of this
genealogy and the events affecting disjoint sections of branches are
independent, so that we can think of $P^\star_{\{ij\}}$, $\{ij\} \in V^{(2)}$,
as a single Poisson point process $P^\star$ on the lineages of pairs of
vertices. A pair of vertices is an edge of $G_{n,r_n}$ if and only if there is
no atom of $P^\star$ on its lineage.

\begin{figure}[h!]
  \centering
  \captionsetup{width=0.7\linewidth}
  \includegraphics[width=0.7\linewidth]{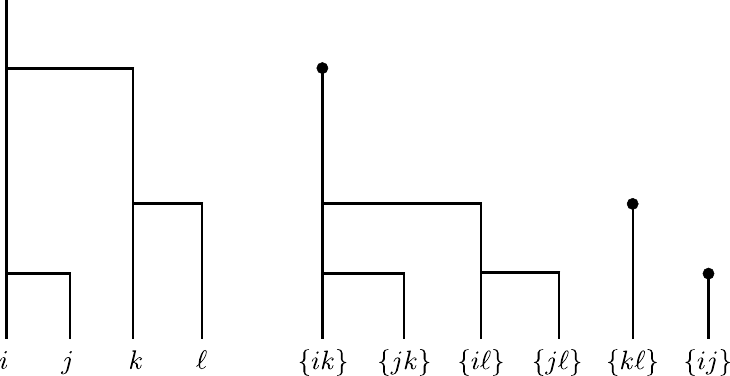}
  \caption[Genealogy of the pairs]
  {On the left, a genealogy on $\Set{i, j, k, \ell}$ and on the right the
  corresponding genealogy of the pairs. Edge removal events occur at constant
  rate along the lineages of pairs of vertices, and a pair of vertices is
  an edge of $G_{n, r_n}$ if and only if there is no atom on its lineage.}
  \label{figGenealogyPairs}
\end{figure}
%]]]

\pagebreak
\subsection{Forward construction} \label{secForwardConstruction}
%[[[
We now give a forward version of the coalescent construction
presented in the previous section. Here, unlike in the previous section,
the graph $G_{n, r}$ is built by adding vertices one at a time.
This construction will be useful in proofs and provides a computationally
efficient way to sample $G_{n, r}$.

Consider the Markov process $(G^\dagger_r(t))_{t \geq 0}$ defined by
\begin{mathlist}
\item $G^\dagger_r(0) = (\Set{1, 2}, \Set{\{1,2\}})$ is the complete graph of order~2.
\item Conditional on $V_t = \Set{1, \ldots, n}$, where $V_t$ is the set of
  vertices of $G^\dagger_r(t)$: at rate $\binom{n}{2}$, a vertex is sampled
  uniformly in $V_t$ and duplicated without replacement -- that is, we copy the
  vertex and all incident edges, and label the new vertex~$n + 1$, resulting in
  a graph with vertex set $\Set{1, \ldots, n + 1}$.
\item During the whole process, each edge disappears at constant rate~$r$.
\end{mathlist}
Next, for every integer ${n \geq 2}$,
let $G^\star_r(n) = G^\dagger_r(t_n-)$, where
\[
  t_n = \sup \Set{t \geq 0 \suchthat
  G^\dagger_r(t) \text{ has } n \text{ vertices}}\,.
\]
Finally, let $\Phi_n(G^\star_r(n))$ denote the graph obtained by shuffling
the labels of the vertices of $G^\star_r(n)$ uniformly at random, i.e.\
let $\Phi_n$ be picked uniformly at random among all the permutations of
$\Set{1, \ldots, n}$ and, by a slight abuse of notation, let
\[
  \Phi_n(G^\star_r(n)) =
    \Big(
      \Set{1, \ldots, n}, \,
      \Set*[\big]{\{\Phi_n(i)\, \Phi_n(j)\} \suchthat \{ij\} \in G^\star_r(n)}
    \Big)
\]

\begin{proposition} \label{propForwardConstruction}
For any $r > 0$, for any integer $n \geq 2$, 
\[
  \Phi_n(G^\star_r(n)) \sim G_{n, r} \, .
\]
\end{proposition}

Going from a backward construction such as
Proposition~\ref{propBackwardWithKingman} to a forward construction such as
Proposition~\ref{propForwardConstruction} is common in coalescent theory. The
proofs, though straightforward, are somewhat tedious. They can be found in
Section~\ref{appBackwardForward} of the Appendix, and we save the rest of this
section to comment on the forward construction.

Proposition~\ref{propForwardConstruction} shows that,
for any given sequence $(r_n)$, for any $n \geq 2$,
$\Phi_n(G^\star_{r_n}(n)) \sim G_{n, r_n}$. Note however that this is \emph{not}
a compatible construction of a sequence
$(G_{n, r_n})_{n \geq 2}$. In particular, all elements of a sequence
$(\Phi_n(G^\star_r(n)))_{n \geq 2}$ are associated to the same value of $r$,
while each term of a sequence $(G_{n, r_n})_{n \geq 2}$ corresponds to a
different value of $r_n$.

Finally, it is necessary to relabel the vertices of $G^\star_r(n)$ in
Proposition~\ref{propForwardConstruction}, as failing to do so
would condition on $\{k, k - 1\}$ being the ${(n - k + 1)}$-th pair of vertices
to coalesce in the genealogy of $G_{n, r}$ (in particular, the edges of
$G^\star_r(n)$ are not exchangeable: ``old'' edges such as $\{1, 2\}$ are
least likely to be present than more recent ones such as $\{n-1, n\}$).
However, since $G^\star_r(n)$ and $\Phi_n(G^\star_r(n))$ are isomorphic,
when studying properties that are invariant under graph isomorphism (such as
the number of connected components in Section~\ref{secConnectedComponents} or
the positive association of the edges in Section~\ref{secNumberOfEdges}), we can
work directly on $G^\star_r(n)$.
%]]]

%]]]

\section{First and second moment methods}
%[[[

%[[[
In this section, we apply Proposition~\ref{propBackwardRepresentation} and
Lemma~\ref{lemmaBackwardConstruction} to obtain the expressions presented in
Table~\ref{tabFirstMoments}. These are then used to identify different regimes
for $G_{n, r_n}$, depending on the asymptotic behavior of the parameter~$r_n$.

In order to be able to use Lemma~\ref{lemmaBackwardConstruction}, we will
always reason conditionally on the genealogy of the vertices (i.e.\ on
the vertex duplication process $\mathpzc{M}$) and then integrate against its
law.
%]]]

\subsection{First moments of graph invariants} \label{secFirstMoments}
%[[[

\subsubsection{Degree and number of edges} \label{secMomentsDegreeEdges}
%[[[

\begin{proposition} \label{propProbaEdge}
For any fixed vertices $i$ and $j$, $i \neq j$, the probability that
$i$ and $j$ are linked in $G_{n, r_n}$ is
\[
  \Prob{i \lkto j} = \frac{1}{1 + r_n} \, .
\]
\end{proposition}

\begin{corollary} \label{colExpectedDegree}
Let $D_n$ be the degree of a fixed vertex of $G_{n, r_n}$, and $\Abs{E_n}$ be
the number of edges of $G_{n, r_n}$. Then,
\[
  \Expec{D_n} = \frac{n - 1}{1 + r_n}
  \quad\text{and}\quad
  \Expec{\Abs{E_n}} = \binom{n}{2} \frac{1}{1 + r_n} \, .
\]
\end{corollary}

\begin{proof}
By Proposition~\ref{propBackwardRepresentation},
\[
  \Set{i \lkto j}  \iff
  P^\star_{\{ij\}} \cap \COInterval[\normalsize]{0, T_{\{ij\}}} = \emptyset \, .
\]
Reasoning conditionally on $T_{\{ij\}}$ and applying
Lemma~\ref{lemmaBackwardConstruction} to $S = \Set{\{ij\}}$ and
$I = \COInterval[\normalsize]{0, T_{\{ij\}}}$, we see that
$P^\star_{\{ij\}}$ is a Poisson point process with rate~$r_n$ on $I$.
Since $T_{\{ij\}} \sim \mathrm{Exp}(1)$,
\[
  \Prob{i \lkto j} = \Prob{e_1 > T_{\{ij\}}} \, , 
\]
where $e_1 = \inf P^\star_{\{ij\}}$ is an exponential variable with rate~$r_n$
that is independent of~$T_{\{ij\}}$.

The corollary follows directly from the fact that the degree of a vertex~$v$
can be written as
\[
  D^{(v)}_n = \sum_{i \neq v}\, \Indic{i \lkto v}
\]
and that the number of edges of $G_{n, r_n}$ is
\[
  \Abs{E_n} = \sum_{\;\{ij\} \in V^{(2)}}\!\!\!\! \Indic{i \lkto j} \, . \qedhere
\]
\end{proof}

\begin{proposition} \label{propCovOverlappingEdges}
Let $i$, $j$ and $k$ be three distinct vertices of $G_{n, r_n}$. We have
\[
  \Cov{\Indic{i \lkto j}, \Indic{i \lkto k}}
  = \frac{r_n}{(3 + 2 r_n)(1 + r_n)^2}
\]
\end{proposition}

\begin{corollary} \label{colVarDegree}
Let $D_n$ be the degree of a fixed vertex of $G_{n, r_n}$. We have
\[
  \Var{D_n} = \frac{r_n (n - 1) (1 + 2\,r_n + n)}{(1 + r_n)^2\,(3 + 2\,r_n)}
\]
\end{corollary}

\begin{proof}
For all $t \geq 0$, let $S_t = \Set{a_t(i), a_t(j), a_t(k)}$.
Let $\tau_1 = \inf \Set{t \geq 0 \suchthat |S_t| = 2}$ and
$\tau_2 = \inf \Set{t \geq \tau_1 \suchthat |S_t| = 1}$. Recall from
Lemma~\ref{propMoran} that $\tau_1$ and $\tau_2 - \tau_1$ are
independent exponential variables with parameter~$3$ and~$1$, 
respectively. Finally, let $\Set{u, v} = S_{\tau_1}$.

By Proposition~\ref{propBackwardRepresentation}, $\{ij\}$ and $\{ik\}$ are
edges of $G_{n, r_n}$ if and only if
$P^\star_{\{ij\}} \cap \COInterval[\normalsize]{0, T_{\{ij\}}}$ and
$P^\star_{\{ik\}} \cap \COInterval[\normalsize]{0, T_{\{ik\}}}$ are empty,
which can also be written
\[
\big(P^\star_{\{ij\}} \cap \COInterval{0, \tau_1}\big) \cup 
\big(P^\star_{\{ik\}} \cap \COInterval{0, \tau_1}\big) \cup 
\big(P^\star_{\{uv\}} \cap \COInterval{\tau_1, \tau_2}\big) = \emptyset
\]
Conditionally on $\tau_1$ and $\tau_2$, by Lemma~\ref{lemmaBackwardConstruction}, 
$(P^\star_{\{ij\}} \cap \COInterval{0, \tau_1}) \cup
(P^\star_{\{ik\}} \cap \COInterval{0, \tau_1})$ is independent of
$P^\star_{\{uv\}} \cap \COInterval{\tau_1, \tau_2}$, $P^\star_{\{ij\}}$ and
$P^\star_{\{ik\}}$ are independent Poisson point processes with rate~$r_n$ on
$\COInterval{0, \tau_1}$, and $P^\star_{\{uv\}}$ is a Poisson point process
with rate~$r_n$ on $\COInterval{\tau_1, \tau_2}$. Therefore,
\[
  \Prob{i \lkto j, i \lkto k} =
  \Prob{e_1 > \tau_1}\,\Prob{e_2 > \tau_2 - \tau_1} \, , 
\]
where $e_1 = \inf (P^\star_{\{ij\}} \cup P^\star_{\{ik\}}) \sim
\mathrm{Exp}(2\,r_n)$ is independent of $\tau_1$ and
$e_2 = \inf (P^\star_{\{uv\}} \cap \COInterval{\tau_1, +\infty}) \sim
\mathrm{Exp}(r_n)$ is independent of $\tau_2 - \tau_1$. As a result,
\[
  \Prob{i \lkto j, i \lkto k} = \frac{3}{3 + 2\,r_n} \times \frac{1}{1 + r_n}\,.
\]
A short calculation shows that
\[
  \Cov{\Indic{i \lkto j}, \Indic{i \lkto k}}
  = \frac{r_n}{(3 + 2 r_n)(1 + r_n)^2} \, ,
\]
proving the proposition.

As before, the corollary follows from writing the degree of $v$ as
$D^{(v)}_n = \sum_{i \neq v}\, \Indic{i \lkto v}$, which gives
\[
  \Var{D^{(v)}_n} = (n - 1) \Var{\Indic{i \lkto v}} +
  (n - 1) (n - 2) \Cov{\Indic{i \lkto v}, \Indic{j \lkto v}} \, .
\]
Substituting $\Var{\Indic{i \lkto v}} = r_n / (1 + r_n)^2$ and
$\Cov{\Indic{i \lkto v}, \Indic{j \lkto v}}$ yields the desired expression.
\end{proof}

\begin{proposition} \label{propCovDisjointEdges}
Let $i$, $j$, $k$ and $\ell$ be four distinct vertices of $G_{n, r_n}$.
We have
\[
  \Cov{\Indic{i \lkto j}, \Indic{k \lkto \ell}} = 
  \frac{2\,r_n}{(1 + r_n)^2 (3 + r_n) (3 + 2\,r_n)}
\]
\end{proposition}

\begin{corollary} \label{colVarEdges}
Let $D^{(i)}_n$ and $D^{(j)}_n$ be the respective degrees of two fixed vertices
$i$ and $j$, and let $\Abs{E_n}$ be the number of edges of $G_{n, r_n}$. We have
\[
  \Cov{D^{(i)}_n, D^{(j)}_n} =
  \frac{r_n}{(1 + r_n)^2} \mleft(1 + \frac{3 (n - 2)}{3 + 2\,r_n} +
  \frac{2 (n - 2)(n - 3)}{(3 + r_n)(3 + 2\,r_n)}\mright)
\]
and
\[
  \Var{\Abs{E_n}} = 
  \frac{r_n\, n\,(n - 1)(n^2 + 2\,r_n^2 + 2\,n\,r_n + n + 5\,r_n + 3)}{2 \,
    (1 + r_n)^2 \, (3 + r_n) \, (3 + 2\,r_n)}
\]
\end{corollary}

The proof of Proposition~\ref{propCovDisjointEdges} and its corollary are
conceptually identical to the proofs of Propositions~\ref{propProbaEdge} and
\ref{propCovOverlappingEdges} and their corollaries, but the calculations are
more tedious and so they have been relegated to
Section~\ref{appFirstMoments} of the Appendix.
%]]]

\subsubsection{Complete subgraphs} \label{secCompleteSubgraphs}
%[[[
From a biological perspective, complete subgraphs are interesting because they
are related to how fine the partition of the set of populations into
species can be. Indeed, the vertices of a complete subgraph -- and especially
of a large one -- should be considered as part of the same species.
A complementary point of view will be brought
by connected components in Section~\ref{secConnectedComponents}.

In this section we establish the following results.

\begin{proposition} \label{propNumberOfCompleteSubgraphs}
Let $X_{n, k}$ be the number of complete subgraphs of order~$k$ in
$G_{n, r_n}$. Then,
\[
  \Expec{X_{n, k}} = \binom{n}{k} \mleft(\frac{1}{1 + r_n}\mright)^{k - 1}.
\]
\end{proposition}

\begin{corollary} \label{colCliqueNumber}
Let $\kappa_n$ be the clique number of $G_{n, r_n}$, i.e.\ the maximal
number of vertices in a complete subgraph of $G_{n, r_n}$.
If $(k_n)$ is such that
\[
  \binom{n}{k_n} \mleft(\frac{1}{1 + r_n}\mright)^{k_n - 1}
  \tendsto{n \to \infty} 0 \, , 
\]
then $k_n$ is asymptotically almost surely an upper bound on $\kappa_n$, 
i.e.\ ${\Prob{\kappa_n \leq k_n} \to 1}$ as $n \to +\infty$.
In particular, when $r_n \to +\infty$,
\begin{mathlist}
\item If $r_n = o(n)$, then
  $\kappa_n \leq \log(r_n) n / r_n$ a.a.s.
\item If $\IsBigOh{r_n}{n/\log(n)}$, $\IsBigOhP{\kappa_n}{n /r_n}$, i.e.\
  \[
    \forall \epsilon > 0, \, \exists M > 0,\, \exists N \st
    \forall n \geq N,\; \Prob{\kappa_n > M n / r_n} < \epsilon \, .
  \]
\end{mathlist}
\end{corollary}

\begin{proof}[Proof of Proposition~\ref{propNumberOfCompleteSubgraphs}]
The number of complete subgraphs of order $k$ of $G_{n, r_n}$ is
\[
  X_{n, k} = \sum_{S \in V^{(k)}} \Indic{G_{n, r_n}[S] \text{ is complete}}
\]
where the elements of $V^{(k)}$ are the $k$-subsets of $V = \Set{1, \ldots, n}$,
and $G_{n, r_n}[S]$ is the subgraph of $G_{n, r_n}$ induced by $S$.
By exchangeability,
\[
  \Expec{X_{n, k}} = \binom{n}{k} \,
  \Prob*{\big}{G_{n, r_n}[S] \text{ is complete}} \,,
\]
where $S$ is any fixed set of $k$ vertices. Using the notation of 
Proposition~\ref{propBackwardRepresentation},
\[
  G_{n, r_n}[S] \text{ is complete} \iff
  \forall \{ij\} \in S, \, P^\star_{\{ij\}} = \emptyset \, .
\]
For all $t \geq 0$, let $A_t = \Set{a_t(i) \suchthat i \in S}$ be the set of
ancestors of $S$ at~$t$. Let $\tau_0 = 0$ and for each $\ell = 1, \ldots, k - 1$ 
let $\tau_\ell$ be the time of the $\ell$-th coalescence between two lineages of
$S$, i.e.\
\[
  \tau_\ell = \inf \Set{t > \tau_{\ell - 1} \suchthat
  \Abs[\normalsize]{A_t} = \Abs[\normalsize]{A_{\tau_{\ell - 1}}} - 1}
\]
Finally, let $\widetilde{A}_\ell = A_{\tau_\ell}$ and
$I_\ell = [\tau_\ell, \tau_{\ell + 1}[$. With this notation,
\[
  \Set{\forall \{ij\} \in S, P^\star_{\{ij\}} = \emptyset} =
  \bigcap_{\ell = 0}^{k - 2} B_{\ell}\,,
\]
where
\[
  B_{\ell} = \!\!\!\bigcap_{\{ij\} \in \widetilde{A}_{\ell}^{(2)}} \!
  \Set*{P^\star_{\{ij\}} \cap I_\ell = \emptyset}
\]
and $\widetilde{A}_{\ell}^{(2)}$ denotes the (unordered)
pairs of $\widetilde{A}_{\ell}$.
Since for $\ell \neq m$, $I_{\ell} \cap I_{m} = \emptyset$,
Lemma~\ref{lemmaBackwardConstruction} shows that conditionally on
$I_0, \ldots, I_{k - 1}$, the events $B_0, \ldots, B_{k - 2}$ are
independent. By construction, for all
$\{ij\} \neq \{uv\}$ in $\widetilde{A}_{\ell}^{(2)}$, 
\[
  \forall t \in I_\ell,\,
  \{a_t(i), a_t(j)\} \neq \{a_t(u), a_t(v)\} \neq \emptyset
\]
and so it follows from Lemma~\ref{lemmaBackwardConstruction} that,
conditional on $I_\ell$,  $(P^\star_{\{ij\}} \cap I_\ell)$,
${\{ij\} \in \widetilde{A}_{\ell}^{(2)}}$, are i.i.d.\ Poisson point processes
with rate $r_n$ on $I_\ell$. Therefore,
\[
  \Prob{B_\ell} = \Prob{\min \Set{ e^{(\ell)}_{\{ij\}} \suchthat
  \{ij\} \in \widetilde{A}_{\ell}^{(2)}} > \Abs{I_\ell}}\, , 
\]
where $e^{(\ell)}_{\{ij\}}$, ${\{ij\} \in \widetilde{A}_{\ell}^{(2)}}$, are
$\binom{k - \ell}{2}$ i.i.d.\ exponential variables with parameter $r_n$ that
are also independent of $\Abs{I_\ell}$. Since $\Abs{I_\ell} \sim
\mathrm{Exp}\mleft(\binom{k - \ell}{2}\mright)$,
\[
  \Prob{B_\ell} = \frac{1}{1 + r_n}
\]
and Proposition~\ref{propNumberOfCompleteSubgraphs} follows.
\end{proof}

\begin{proof}[Proof of Corollary~\ref{colCliqueNumber}]
The first part of the corollary is a direct 
consequence of Proposition~\ref{propNumberOfCompleteSubgraphs}. First, note that
\[
  X_{n,k_n} = 0 \;\iff\; \kappa_n < k_n
\]
that a complete subgraph of order~$k$ contains complete subgraphs of order
$\ell$ for all $\ell < k$. As a result, any $k_n$ such that
$\Prob{X_{n,k_n} = 0} \to 1$ is asymptotically almost surely an upper bound on
the clique number $\kappa_n$. Now, observe that since
$X_{n,k_n}$ is a non-negative integer,
$X_{n,k_n} \geq \Indic{X_{n,k_n} \neq 0}$ and therefore
\[
  \Expec{X_{n,k_n}} \geq \Prob{X_{n,k_n} \neq 0} \,.
\]
Finally, $X_{n, k}$ being integer-valued, $\Prob{X_{n,k_n} \neq 0} \to 0$
implies $\Prob{X_{n,k_n} = 0} \to~1$.

To prove the second part of the corollary, using Stirling's formula we find that
whenever $r_n$ and $k_n$ are $o(n)$ and go to $+\infty$ as $n \to +\infty$,
\[
  \binom{n}{k_n} \mleft(\frac{1}{1 + r_n}\mright)^{k_n - 1}
  \sim \;
  \frac{C}{\sqrt{k_n}} \frac{n^n}{k_n^{k_n} (n - k_n)^{n - k_n}}
  \mleft(\frac{1}{1 + r_n}\mright)^{k_n - 1} \!\!\!\!, 
\]
where $C = \sqrt{2\pi}$. The right-hand side goes to zero if and only if
its logarithm goes to $-\infty$, i.e.\ if and only if
\[
  A_n \;\defas\; k_n \log\mleft(\frac{n - k_n}{k_n (1 + r_n)}\mright)
  - n \log\mleft({1 - \frac{k_n}{n}}\mright)
  + \log\mleft(\frac{1 + r_n}{\sqrt{k_n}}\mright)
\]
goes to $-\infty$. Now let $k_n = n g_n / r_n$, where $g_n \to +\infty$ and is 
$o(r_n)$, so that $k_n = o(n)$. Then,
\[
  k_n \log\mleft(\frac{n - k_n}{k_n (1 + r_n)}\mright) \sim - k_n \log(g_n)
\]
and
\[
  - n \log\mleft({1 - \frac{k_n}{n}}\mright) \sim k_n \,.
\]
Moreover, as long as it does not go to zero,
\[
  \log\mleft(\frac{1 + r_n}{\sqrt{k_n}}\mright) \sim
  \frac{3}{2} \log(r_n) - \frac{1}{2} \log(n g_n) \, .
\]
Putting the pieces together, we find that $A_n$ is asymptotically equivalent to
\[
  - \frac{n g_n}{r_n} \log(g_n)
  + \frac{3}{2} \log(r_n) - \frac{1}{2} \log(n g_n) \, .
\]
Taking $g_n = \log(r_n)$, this expression goes to $-\infty$ as $n \to +\infty$,
yielding (i). If $\IsBigOh{r_n}{n / \log(n)}$, then it goes to
$-\infty$ for any $g_n \to +\infty$, which proves~(ii). Indeed,
if there exists $\epsilon > 0$ such that
\[
  \forall M > 0,\; \forall N,\; \exists n \geq N \st
  \Prob{\kappa_n > M n / r_n} \geq \epsilon \, , 
\]
then considering successively $M = 1, 2, \ldots$, we can find
$n_1 < n_2 < \cdots$ such that
\[
  \forall k \in \N, \;
  \Prob{\kappa_{n_k} > k n_k / r_{n_k}} \geq \epsilon \,.
\]
Defining $(g_n)$ by
\[
  \forall n \in \Set{n_k, \ldots, n_{k + 1} - 1}, \;
  g_n = k \, , 
\]
we obtain a sequence $(g_n)$ that goes to infinity and yet is such that for all
$N$ there exists $n \defas \min\Set{n_k \suchthat n_k \geq N}$ such that
$\Prob{\kappa_n > g_n n / r_n} \geq \epsilon$.
\end{proof}

A natural pendant to Proposition~\ref{propNumberOfCompleteSubgraphs} and
Corollary~\ref{colCliqueNumber} would be to use the variance of $X_{n, k}$
to find a lower bound on the clique number.
Indeed, it follows from Chebychev's inequality that
\[
  \Prob{X_{n, k} = 0} \leq \frac{\Var{X_{n, k}}}{\Expec{X_{n, k}}^2} \, .
\]
However, computing $\Var{X_{n, k}}$ requires being able to compute the
probability that two subsets of $k$ vertices $S$ and $S'$ both induce a complete
subgraph, which we have not managed to do. Using the probability that
$G_{n, r_n}[S]$ is complete as an upper bound for this quantity, we have the
very crude inequality
\[
  \Var{X_{n, k}} \leq \binom{n}{k}^2 p\, (1 - p)\,,
\]
where $p = 1/(1 + r_n)^{k - 1}$. This shows that when $r_n \to 0$ and
$k_n = o(1/r_n)$, $\Prob{X_{n, k_n} = 0}$ tends to zero, proving that
$\kappa_n$ is at least $\Theta(1/r_n)$.

Finally, because we expect our model to form dense connected components,
whose number we conjecture to be on the order of $r_n$ in the intermediate
regime (see Theorem~\ref{thmConnectedComponents} and
Conjecture~\ref{conjConnectedComponents}), and since the degree of a typical
vertex is approximately $n/r_n$ in that regime, it seems reasonable to
conjecture
\begin{conjecture} \label{conjCompleteSubgraphs} 
In the intermediate regime, i.e.\ when $r_n \to +\infty$ and $r_n = o(n)$,
\[
 \exists \alpha, \beta > 0 \st 
 \Prob{\alpha n/r_n \leq \kappa_n \leq \beta n/r_n} \tendsto{n\to+\infty} 1.
\]
\end{conjecture}
%]]]

%]]]

\subsection{Identification of different regimes} \label{secRegimes}
%[[[

We now use the results of the previous section to identify different regimes
for the behavior of $G_{n,r_n}$. The proof of our next theorem relies in
part on results proved later in the paper (namely,
Theorems~\ref{thmDegreeDistribution} and~\ref{thmPoissonianEdges}), but no
subsequent result depends on it, avoiding cyclic dependencies.
While this section could have been placed at the end of the paper,
it makes more sense to present it here because it relies mostly on
Section~\ref{secFirstMoments} and because it helps structure the rest of the
paper.

\begin{theorem} \label{thmRegimes}
Let $D_n$ be the degree of a fixed vertex of $G_{n, r_n}$. In the limit as
$n \to +\infty$, depending on the asymptotics of $r_n$ we have the following
behaviors for~$G_{n, r_n}$
\begin{mathlist}
\item \emph{Transition for the complete graph:} when $\IsSmallOh{r_n}{1/n}$, 
  $\Prob{G_{n, r_n} \text{ is complete}}$ goes to~$1$,
  while when $\IsSmallOmega{r_n}{1/n}$ it goes to~$0$;
  when $\IsBigTheta{r_n}{1/n}$, this probability is bounded away from~$0$
  and from~$1$.
\item \emph{Dense regime:} when $\IsSmallOh{r_n}{1}$, $\Prob{D_n = n - 1} \to 1$.
\item \emph{Sparse regime:} when $\IsSmallOmega{r_n}{n}$,
  $\Prob{D_n = 0} \to 1$.
\item \emph{Transition for the empty graph:} when $\IsSmallOh{r_n}{n^2}$, 
  $\Prob{G_{n, r_n} \text{ is empty}}$ goes to~$0$
  while when $\IsSmallOmega{r_n}{n^2}$ it goes to~$1$;
  when $\IsBigTheta{r_n}{n^2}$, this probability is bounded away from~$0$ and
  from~$1$.
\end{mathlist}
\end{theorem}

\begin{proof}
{(i)} is a direct consequence of
Proposition~\ref{propNumberOfCompleteSubgraphs} which, applied to $k = n$,
yields
\[
  \Prob{G_{n, r_n} \text{ is complete}} =
  \mleft(\frac{1}{1 + r_n}\mright)^{n - 1} \!.
\]

{(ii)} is intuitive since $\Expec{D_n} = (n - 1)/(1 + r_n)$; but
because it takes $\IsSmallOh{r_n}{1/n^2}$ for $\Var{D_n}$ to go to zero, a second
moment method is not sufficient to prove it.
However, using Theorem~\ref{thmDegreeDistribution}, we see that
$\Prob{D_n = n - 1}$ can be written as
\[
  \Prob{D_n = n - 1} =
  \frac{\Gamma(2 + 2\, r_n) \Gamma(n + 1)}{\Gamma(n + 1 + 2\, r_n)} \, , 
\]
where $\Gamma$ is the gamma function. The results follows by letting
$r_n$ go to zero and using the continuity of $\Gamma$.

{(iii)} follows from the same argument as in the proof of
Corollary~\ref{colCliqueNumber}, by which, $D_n$ being a non-negative integer,
$\Prob{D_n \neq 0} \leq \Expec{D_n} = \frac{n - 1}{1 + r_n}$.

In {(iv)}, the fact that $G_{n, r_n}$ is empty when $r_n = \omega(n^2)$ is
yet another application of this argument, but this time using the
expected number of edges, $\Expec{|E_n|} = \frac{n (n - 1)}{2 (1 + r_n)}$, in
conjunction with the fact that $G_{n, r_n}$ is empty if and only if $|E_n| = 0$;
to see why the graph cannot be empty when $r_n = o(n^2)$, consider the edge
that was created between the duplicated vertex and its copy in the most recent
duplication. Clearly, if this edge has not disappeared yet then $G_{n, r_n}$
cannot be empty. But the probability that this edge has disappeared is just
\[
  \frac{r_n}{\binom{n}{2} + r_n}\, , 
\]
which goes to zero when $r_n = o(n^2)$. Finally,
the fact that  $\Prob{G_{n, r_n} \text{ is empty}}$ is bounded away from
$0$ and from $1$ when $r_n = \Theta(n^2)$ is a consequence of
Theorem~\ref{thmPoissonianEdges}, which shows 
that the number of edges is Poissonian
when $r_n = \omega(n)$. As a result, $\Prob{|E_n| = 0} \sim e^{-\Expec{|E_n|}}$.
\end{proof}

\begin{remark} Note that
when $\IsSmallOh{r_n}{1}$, $\Var{D_n} \sim r_n n^2 / 3$ can go to infinity even
though $D_n = n - 1$ with probability that goes to~1. Similarly, when
$\IsSmallOh{r_n}{1/n}$, $\Var{|E_n|} \sim r_n n^4 / 18$ and $|E_n| =
\binom{n}{2}$ a.a.s. Notably, $\,\overline{\!D_n\!}\, = (n - 1) - D_n$
converges to $0$ in probability while $\Var{\,\overline{\!D_n\!}\,}$ goes to
infinity.
\end{remark}
%]]]

%]]]

\section{The degree distribution} \label{secDegreeDistribution}
%[[[

%[[[
The degree distribution is one of the most widely studied graph invariants in
network science. Our model makes it possible to obtain an exact expression for
its probability distribution:

\begin{theorem}[degree distribution] \label{thmDegreeDistribution}
Let $D_n$ be the degree of a fixed vertex of $G_{n, r_n}$. Then,  
for each $k \in \Set{0, \ldots, n - 1}$,
\[
  \Prob{D_n = k} = \frac{2\,r_n\,(2\,r_n + 1)}{(n + 2\,r_n) (n - 1 + 2\,r_n)} \,
  (k + 1) \, \prod_{i = 1}^k \frac{n - i}{n - i + 2\,r_n - 1}\, , 
\]
where the empty product is 1.
\end{theorem}

The expression above holds for any positive sequence $(r_n)$ and any $n$; but
as $n \to +\infty$ it becomes much simpler and, under appropriate rescaling,
the degree converges to classical distributions:

\begin{theorem}[convergence of the rescaled degree] \label{thmDegreeConvergence} ~
\begin{mathlist}
  \item If $r_n \to r > 0$, then $\frac{D_n}{n}$ converges in distribution to a 
    $\mathrm{Beta}(2, 2\,r)$ random variable.
  \item If $r_n$ is both $\SmallOmega{1}$ and $\SmallOh{n}$, then
    $\frac{D_n}{n / r_n}$ converges in distribution to a size-biased
    exponential variable with parameter~$2$.
  \item If $2\,r_n / n \to \rho > 0$, then $D_n + 1$ converges in distribution
    to a size-biased geometric variable with parameter $\rho/(1 + \rho)$.
\end{mathlist}
\end{theorem}

In this section we prove Theorem~\ref{thmDegreeDistribution} by coupling the
degree to the number of individuals descended from a founder in a branching
process with immigration. Theorem~\ref{thmDegreeConvergence} is then easily
deduced by a standard study that has been relegated to
Section~\ref{appDegreeConvergence} of the Appendix.
%]]]

\subsection{Ideas of the proof of Theorem~\ref*{thmDegreeDistribution}} \label{secIdeasProofDegree}
%[[[
Before jumping to the formal proof of Theorem~\ref{thmDegreeDistribution}, we
give a verbal account of the main ideas of the proof.

In order to find the degree of a fixed vertex $v$, we have to consider all
pairs $\{iv\}$ and look at their ancestry to assess the absence/presence of
atoms in the corresponding Poisson point processes. To do so, we can
restrict our attention to the genealogy of the vertices, and consider that
edge-removal events occur along the lineages of this genealogy: a point that
falls on the lineage of vertex~$i$ at time~$t$ means that
$t \in P^\star_{\{iv\}}$. In this setting, edge-removal events occur at
constant rate $r_n$ on every lineage different from that of $v$.

Next, the closed neighborhood of $v$ (i.e.\ the set of vertices that are linked
to $v$, plus $v$ itself) can be obtained through the following procedure:
we trace the genealogy of vertices, backwards in time; if we encounter
an edge-removal event on lineage~$i$ at time~$t$, then we mark all vertices that 
descend from this lineage, i.e.\ all vertices whose ancestor at time~$t$ is
$i$; only the lineages of unmarked vertices are considered after~$t$.
We stop when there is only one lineage left in the genealogy.
The unmarked vertices are then exactly the neighbors of $v$ (plus $v$ itself).
The procedure is illustrated in Figure~\ref{figProofDegree}.

% Just to make it fit in the right place; should be removed for final submission
\enlargethispage{2ex}
\begin{figure}[h!]
  \centering 
  \includegraphics[width=\linewidth]{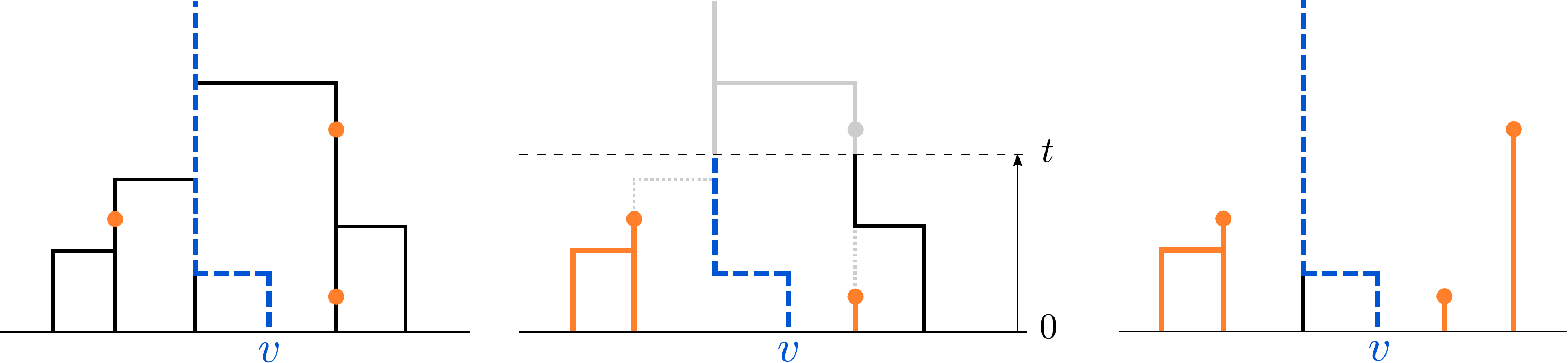}
  \caption{Illustration of the procedure used to find the neighborhood of $v$.
  On the left, the genealogy of the vertices. The dashed blue line represents
  the lineage of the focal vertex $v$, and a dot on lineage~$k$ corresponds to
  a point in $P_{\{k a_t(v)\}}$. In the middle, we uncover the genealogy and
  edge-removal events in backward time, as described in the main text. On the
  right, the forest that we get when the procedure is complete. The non-colored
  (black) branches are exactly the neighbors of $v$.}
  \label{figProofDegree}
\end{figure}

This vertex marking process is not convenient to describe in backward time
because we typically mark several vertices simultaneously.  By contrast, the
forest that results from the completed process seems much easier to describe in
forward time. Indeed, the arrival of a new lineage corresponds either to the
addition of a new unmarked vertex or to the addition of a marked one, depending
on whether the new lineage belongs to the same tree as $v$ or not.

Moreover, in forward time, the process is reminiscent of a branching process
with immigration: new lineages are either grafted to existing ones (branching)
or sprout spontaneously (immigration). Let us try to find what the branching
and immigration rates should be. In backward time, when there are $k+1$
lineages then a coalescence occurs at rate $\binom{k + 1}{2}$, while an
edge-removal event occurs at rate $k\,r_n$. Reversing time, these events occur
at the same rates. As a result, when going from $k$ to $k+1$ lineages, the
probability that the next event is a branching is $(k + 1) / (k + 1 + 2\,r_n)$.

\enlargethispage{1ex}

Next, we have to find the probability that each lineage has to branch,
given that the next event is a branching. Here, a spinal
decomposition~\cite{ChauvinRouaultSPA1991, LyonsPemantlePerezAnnProbab1995}
suggests that every lineage branches at rate~$1$, except for the lineage
of $v$, which branches at rate~$2$. To see why, observe that this is
coherent with the fact that, in backward time, when going from $k + 1$ to $k$
lineages there are $k$ pairs out of $\binom{k + 1}{2}$ that involve the lineage
of $v$, so that the probability that the lineage of $v$ is involved in the next
coalescence is $2/(k + 1)$.

If this heuristic is correct, then in forward time it is easy to track the
number of branches of the tree of $v$ versus the number of branches of other
trees: when there are $p$ branches in the tree of $v$ and $q$ branches in the
other trees, the probability that the next branch is added to the tree of
$v$ is just $(p + 1) / (p + 1 + q + 2\,r_n)$. Moreover, when the total number
of branches reaches $n$, the number of branches in the tree of $v$ is also the
number of unmarked vertices at the end of the vertex marking procedure, which
is itself $D_n^{(v)} + 1$, the degree of $v$ plus one.
%]]]

\subsection{Formal proof of Theorem~\ref*{thmDegreeDistribution}} \label{secFormalProofDegree}
%[[[

%[[[
The ideas and outline of the proof parallels the account given in the previous
section: first, given a realization of the vertex-duplication process
$\mathpzc{M}$ and of the edge-removal process $\mathpzc{P}$, we describe a
deterministic procedure that gives the closed neighborhood of any vertex $v$, 
\[
  N_G[v] = \Set*[\big]{i \in V \suchthat \{iv\} \in E} \cup \Set*[\big]{v}\,, 
\]
where $G = (V, E)$ is the graph associated to $\mathpzc{M}$ and $\mathpzc{P}$;
then, we identify the law of the process $(F_t)_{t \geq 0}$ corresponding
to this procedure, and recognize it as the law of a branching process with
immigration. 

\begin{definition} \label{defRootedForestMarkedVertices}
A \emph{rooted forest with marked vertices} is a triple
$F = (\unmarked{V}, \marked{V}, \vec{E})$ such that 
\begin{mathlist}
  \item $\unmarked{V} \cap \marked{V} = \emptyset$.
  \item Letting $V = \unmarked{V} \cup \marked{V}$, $(V, \vec{E})$ is an acyclic
    digraph with the property that $\forall i \in V$,
    $\mathrm{deg}^+(i) \in \{0, 1\}$, where
    $\mathrm{deg}^+(i)$ is the out-degree of vertex~$i$.
\end{mathlist}
The marked vertices are the elements of $\marked{V}$; the roots of $F$ are
the vertices with out-degree 0 (that is, edges are oriented towards the root),
whose set we denote by $R(F)$; finally, the trees of $F$ are its
connected components (in the weak sense, i.e.\ considering the underlying
undirected graph), and we write $T_F(i)$ for the tree containing $i$ in $F$.
\end{definition}
%]]]

\newpage
\newgeometry{textwidth=390pt, top=2.0cm, bottom=2.6cm}
\enlargethispage{0.4cm}

\subsubsection{The vertex-marking process}
%[[[
We now define the backward-time process $(F_t)_{t \geq 0}$ that corresponds to
the procedure described informally in Section~\ref{secIdeasProofDegree}.
Recall the notation of Proposition~\ref{propBackwardRepresentation}.
For a given realization of $\mathpzc{M}$ and $\mathpzc{P}$, and for any fixed
vertex~$v$, let $(F_t)_{t \geq 0}$ be the piecewise constant process
defined deterministically by
\begin{itemize}
  \item $F_0 = (V, \emptyset, \emptyset)$.
  \item If $t \in M_{(ij)}$, then
    $\forall k, \ell \in R(F_{t-}) \cap \unmarked{V}_{t-}$
    such that $(a_{t-}(k), a_{t-}(\ell)) = (i, j)$,
    \[
      \vec{E}_t = \vec{E}_{t-} \cup \big\{(k, \ell)\big\} \, .
    \]
  \item If $t \in P_{\{ia_t(v)\}}$, then letting
    $d_{t}(i) = \Set{j \in V \suchthat a_{t}(j) = i}$ be the set of descendants
    of $i$ born after time~$t$,
    \[
      \begin{dcases}
        \unmarked{V}_t = \unmarked{V}_{t-} \setminus d_{t}(i) \\
        \marked{V}_t = \marked{V}_{t-} \cup d_{t}(i) \, .
      \end{dcases}
    \]
\end{itemize}

What makes $(F_t)_{t \geq 0}$ interesting is that
\[
  N_G[v] = \unmarked{V}_{\infty} \, .
\]
Indeed, by construction,
\[
  i \in \unmarked{V}_t \iff
  \bigcup_{s \in [0, t]}
  \mleft( \mspace{-6mu}
  \bigcup_{\mspace{16mu}j : i \in d_s(j)} \mspace{-16mu} P_{\{ja_s(v)\}} \mright)
  = \emptyset \,, 
\]
and since for every $s$ the unique $j$ such that $i \in d_s(j)$ is $a_s(i)$,
we have
\[
  \unmarked{V}_t =
  \Set{i \in V \suchthat P^\star_{\{iv\}} \cap [0, t] = \emptyset} \, .
\]

The Poissonian construction given above shows that $(F_t, a_t)_{t \geq 0}$
is a Markov process. Now, observe that conditional on $a_t$
\begin{mathlist}
\item $M_{(ij)} \cap \OpenInterval{t, +\infty}
       \;\sim\;
       M_{(a_t(i)a_t(j))} \cap \OpenInterval{t, +\infty}$
       and is independent of $(F_s, a_s)_{s \leq t}$
\item $P_{\{ia_t(v)\}} \cap \OpenInterval{t, +\infty}
       \;\sim\;
       P_{\{a_t(i)a_t(v)\}} \cap \OpenInterval{t, +\infty}$
       and is independent of $(F_s, a_s)_{s \leq t}$
\item $j \in d_t(i) \iff i \in R(F_t) \text{ and } j \in T_{F_t}(i)$
\end{mathlist}
As a consequence, $(F_t)_{t \geq 0}$ is also a Markov process, whose law is
characterized~by
\begin{itemize}
  \item $F_0 = (V, \emptyset, \emptyset)$.
  \item $F_t$ goes from $(\unmarked{V}_t, \marked{V}_t, \vec{E}_t)$ to
    \begin{itemize}
      \item $\big(\unmarked{V}_t,\, \marked{V}_t,\, \vec{E}_t \cup \{(i, j)\}\big)$
        at rate 1/2, for all $i, j$ in $R(F_t)$
      \item $\big(\unmarked{V}_t \setminus T_{F_t}(i),\, \marked{V}_t \cup
        T_{F_t}(i),\, \vec{E}_t\big)$ at rate $r_n$, for all $i$ in $R(F_t)$.
    \end{itemize}
\end{itemize}

\newpage
\newgeometry{textwidth=390pt, top=2.4cm, bottom=2.8cm}

Let $(\widetilde{F}_k)_{k \in \{1, \ldots, n\}}$ be the chain
embedded in $(F_t)_{t \geq 0}$, i.e.\ defined by
\[
  \widetilde{F}_k = F_{t_k}, \;
  \text{where }
  t_k = \inf \Set*[\big]{t \geq 0 \suchthat \Abs{R(F_t)} = n - k + 1} \,.
\]
The rooted forests with marked vertices that correspond to realizations of
$\widetilde{F}_n$ are exactly the $f_n = (\unmarked{V}, \marked{V}, \vec{E})$ 
that have $n$ vertices and are such that $\unmarked{V} = T_{f_n}(v)$. Moreover,
for each of these there exists a unique trajectory $(f_1, \ldots, f_n)$ 
of $(\widetilde{F}_1, \ldots, \widetilde{F}_n)$ such that
$\widetilde{F}_n = f_n$ and it follows from the transition rates of
$(F_t)_{t \geq 0}$ that 
\begin{align} \label{eqLawVertexMarking}
  \Prob{\widetilde{F}_n = f_n}
  &= \frac{(1/2)^{n - |R(f_n)| }\; r_n^{|R(f_n)| - 1}}{
    \prod\limits_{k = 2}^n (k(k - 1) / 2 + (k - 1) r_n)} \nonumber \\
  &= \frac{1}{(n - 1)!} \times \frac{(2\,r_n)^{|R(f_n)| - 1}}{
    \prod\limits_{k = 2}^n (k + 2\,r_n)}
\end{align}
Finally, note that $\unmarked{\widetilde{V}}_n = \unmarked{V}_\infty$ is the
closed neighborhood of $v$ in our graph.
%]]]

\subsubsection{The branching process}
%[[[
The process with which we will couple the vertex-marking process described in
the previous section is a simple random function of the trajectories
of a branching process with immigration $(Z_t)_{t \geq 0}$.
In this branching process, immigration occurs at rate $2\,r_n$
and each particle gives birth to a new particle at rate $1$ -- except for one
particle, which carries a special item that enables it to give birth
at rate $2$; when this lineage reproduces, it keeps the item with probability
$1/2$, and passes it to its offspring with probability $1/2$.

Formally, we consider the Markov process on the set of rooted forests with
marked vertices (augmented with an indication of the carrier of the item),
defined by $Z_0 = (\{1\}, \emptyset, \emptyset, 1)$ and
by the following transition rates:

$(Z_t)_{t \geq 0}$ goes from $(\unmarked{W},\, \marked{W},\, \vec{E},\, c)$ to
\begin{itemize}
  \item $\big(\unmarked{W} \cup \{N\},\; \marked{W},\;
    \vec{E} \cup \{(N, i)\},\; c\big)$
    at rate~$1$, for all $i \in \unmarked{W}$
  \item $\big(\unmarked{W},\; \marked{W} \cup \{N\},\;
    \vec{E} \cup \{(N, i)\}, \; c\big)$
    at rate~$1$, for all $i \in \marked{W}$
  \item $\big(\unmarked{W} \cup \{N\},\; \marked{W},\;
    \vec{E} \cup \{(N, c)\},\; N\big)$
    at rate~$1$
  \item $\big(\unmarked{W},\; \marked{W} \cup \{N\},\; \vec{E},\; c \big)$
    at rate~$2\,r_n$ 
\end{itemize}
where $N = \Abs{\unmarked{W} \cup \marked{W}} + 1$ is the label of the new
particle. The fourth coordinate of $(Z_t)_{t \geq 0}$ tracks the carrier of
the item.

As previously, the Markov chain $(\widetilde{Z}_k)_{k \in \N^*}$ embedded in
$(Z_t)_{t \geq 0}$ is defined by
\[
  \widetilde{Z}_k = Z_{t_k}, \;
  \text{where }
  t_k = \inf \Set*[\big]{t \geq 0 \suchthat
  \Abs{\unmarked{W}_t \cup \marked{W}_t} = k}\, .
\]
The realizations of $\widetilde{Z}_n$ are exactly the 
$(\unmarked{W}_n, \marked{W}_n, \vec{E}_n, c_n)$ such that $f_n =
(\unmarked{W}_n, \marked{W}_n, \vec{E}_n)$ is a rooted forest with marked
vertices on $\Set{1, \ldots, n}$ and $\unmarked{W}_n = T_{f_n}(1) =
T_{f_n}(c_n)$. For these, it follows from the transition rates of
$(Z_t)_{t \geq 0}$ that
\begin{equation} \label{eqLawBPImmigration}
  \Prob{\widetilde{Z}_n = (\unmarked{W}_n, \marked{W}_n, \vec{E}_n, c_n)} =
  \frac{(2\,r_n)^{|R(f_n)| - 1}}{\prod\limits_{k = 1}^{n - 1}(k + 1 + 2\,r_n)}\,.
\end{equation}

Finally, note that $(X_k)_{k \in \N^*} = \big(|\unmarked{\widetilde{W}}_k|,\,
|\marked{\widetilde{W}}_k|\big)_{k \in \N^*}$, which counts the number of
descendants of the first particle and the number of descendants of immigrants,
is a Markov chain whose law is characterized by $X_1 = (1, 0)$ and $X_k$ goes
from $(p, q)$ to
\begin{itemize}
  \item $(p + 1, q)$ with probability $\frac{p + 1}{p + 1 + q + 2r_n}$
  \item $(p, q + 1)$ with probability $\frac{q + 2r}{p + 1 + q + 2r_n}$.
\end{itemize}
%]]]

\subsubsection{Relabeling and end of proof}
%[[[
The last step before finishing the proof of Theorem~\ref{thmDegreeDistribution}
is to shuffle the vertices of the forest associated to $\widetilde{Z}_n$
appropriately. For any fixed $n$, $v$ and $c$ in $\{1, \ldots, n\}$, let
$\Phi_{(c, v)}$ be uniformly and independently of anything else picked among
the permutations of $\{1, \ldots, n\}$ that map $c$ to $v$; define
$\Phi_v(\widetilde{Z}_n)$ by
\[
  \Phi_v\mleft(\unmarked{\widetilde{W}}_n, \marked{\widetilde{W}}_n,
  \widetilde{E}_n, \widetilde{c}_n\mright) = 
  \mleft(\Phi_{(\tilde{c}_n, v)}(\unmarked{\widetilde{W}}_n),\;
  \Phi_{(\tilde{c}_n, v)}(\marked{\widetilde{W}}_n),\;
  \Phi_{(\tilde{c}_n, v)}(\widetilde{E}_n)
  \mright)
\]
where $\Phi_{(\tilde{c}_n, v)}(\widetilde{E}_n)$ is to be understood as
$\Set{\big(\Phi_{(\tilde{c}_n, v)}(i), \Phi_{(\tilde{c}_n, v)}(j)\big) \suchthat
(i, j) \in \widetilde{E}_n}$.

With all these elements, the proof of Theorem~\ref{thmDegreeDistribution} goes
as follows. First, from equations~\eqref{eqLawVertexMarking} and
\eqref{eqLawBPImmigration} and the definition of $\Phi_v$,
we see that for all rooted forest with marked vertices $f_n$, 
\[
  \Prob{\widetilde{F}_n = f_n} =
  \Prob{\Phi_v(\widetilde{Z}_n) = f_n} \, .
\]
In particular, $\unmarked{\widetilde{V}}_n$, the set of unmarked vertices
in the vertex-marking process,
and $\Phi_{(\tilde{c}_n,v)}(\unmarked{\widetilde{W}}_n)$,
the relabeled set of descendants of the first particle in the branching
process, have the same law.  Now, on the one hand we have
\[
  \Abs{\unmarked{\widetilde{V}}_n} = \Abs[\big]{N_G[v]} = D^{(v)}_n + 1 \, , 
\]
and on the other hand we have
\[
  \Abs{\Phi_{(\tilde{c}_n,v)}(\unmarked{\widetilde{W}}_n)} =
  \Abs{\unmarked{\widetilde{W}}_n} \, . 
\]

Since $\Abs{\unmarked{\widetilde{W}}_n}$
is the first coordinate of the Markov chain $(X_k)_{k \in \N^*}$
introduced in the previous section, it follows directly from the
transition probabilities of $(X_k)_{k \in \N^*}$ that
\[
  \Prob*{\big}{X_n = (k + 1, n - k - 1)} =
 \binom{n - 1}{k} \, \frac{\prod\limits_{p = 1}^{k}(p + 1) \,
  \prod\limits_{q = 0}^{n - k - 2}
  (q + 2\,r_n)}{\prod\limits_{(p+q) = 1}^{n - 1} \!
  \big((p\!+\!q) + 1 + 2\,r_n\big)} \, , 
\]
from which the expression of Theorem~\ref{thmDegreeDistribution} can be
deduced through elementary calculations.
%]]]

%]]]

%]]]

\section{Connected components in the intermediate regime} \label{secConnectedComponents}
%[[[

%[[[
From a biological perspective, connected components are good candidates to
define species, and have frequently been used to that end.
Moreover, among the possible definitions of species, they play a special role
because they indicate how coarse the partition of the set of populations into
species can be; indeed, it would not make sense biologically for distinct
connected components to be part of the same species. As a result, connected
components are in a sense the ``loosest'' possible definition of species.  This
complements the perspective brought by complete subgraphs, which inform us on
how fine the species partition can be (see Section~\ref{secCompleteSubgraphs}).
For a discussion of the definition of species in a context where traits and
ancestral relationships between individuals are known,
see~\cite{ManceauLambert2016}.

The aim of this section is to prove the following theorem.

\begin{theorem} \label{thmConnectedComponents}
Let $\nbcc_n$ be the number of connected components of $G_{n, r_n}$. If
$r_n$ is both $\SmallOmega{1}$ and $\SmallOh{n}$, then
\[
  \frac{r_n}{2} + \SmallOhP{r_n}
  \leq \; \nbcc_n \; \leq
  2\,r_n \log n + \SmallOhP{r_n \log n}
\]
where, for a positive sequence $(u_n)$, $\SmallOhP{u_n}$ denotes a given
sequence of random variables $(X_n)$ such that $X_n/u_n \to 0$ in probability.
\end{theorem}
%]]]

\subsection{Lower bound on the number of connected components}
%[[[
The proof of the lower bound on the number of connected components
uses the forward construction introduced in
Section~\ref{secBackwardConstruction} and the associated notation. It relies
on the simple observation that, letting $\nbcc(G)$ denote the number of
connected components of a graph $G$, $\nbcc(G^\star_{r_n}(k))$ is a
nondecreasing function of $k$. Indeed, in the sequence of events defining
$(G^\star_{r_n}(k))_{k \geq 2}$, vertex duplications do not change the number
of connected components -- because a new vertex is always linked to an existing
vertex (its `mother') and her neighbors -- and edge removals can only increase
it.  Thus, if $m_n \leq n$ and $\ell_n$ are such that
$\Prob{\nbcc(G^\star_{r_n}(m_n) \geq \ell_n} \to 1$ as $n \to \infty$, then
$\ell_n$ is asymptotically almost surely a lower bound on the number of
connected components of $G^\star_{r_n}(n)$ --- and therefore of $G_{n, r_n}$.

To find such a pair $(m_n, \ell_n)$, note that, for every graph $G$ of order
$m$, 
\[
  \nbcc(G) \geq m - \#\mathrm{edges}(G) \, .
\]
Moreover, since for any fixed $n$, $G^\star_{r_n}(m_n)$ has the same law as
$G_{m_n, r_n}$, the exact expressions for the expectation and the variance of
$|E^\star_{m_n}|$, the number of edges of $G^\star_{r_n}(m_n)$, are given in
Table~\ref{tabFirstMoments}. We see that, if $r_n$ and $m_n$ are both
$\SmallOmega{1}$ and $\SmallOh{n}$,
\[
  \Expec{|E^\star_{m_n}|} \sim \frac{m_n^2}{2\,r_n}
  \quad\text{and}\quad
  \Var{|E^\star_{m_n}|} \sim \frac{m_n^2}{4\,r_n^3} \big(m_n^2 + 2\,r_n^2\big)\,.
\]
By Chebychev's inequality,
\[
  \Prob*{\Big}{\,\big| |E^\star_{m_n}| - \Expec{|E^\star_{m_n}|} \big| \,
    \geq m_n^{1 - \epsilon}}
  \; \leq \;
  \frac{\Var{|E^\star_{m_n}|}}{m_n^{2 - 2\epsilon}} \, .
\]
When $\IsBigTheta{m_n}{r_n}$, since $r_n = \SmallOmega{1}$ the right-hand
side of this inequality goes to $0$ as $n \to +\infty$,
for all $\epsilon < 1/2$. It follows that
\[
  |E^\star_{m_n}| = \Expec{|E^\star_{m_n}|} + \SmallOhP{r_n} \, .
\]
Taking $m_n \defas \lfloor \alpha\,r_n \rfloor$, we find that
\[
  \nbcc(G^\star_{r_n}(m_n)) \geq m_n - |E^\star_{m_n}| =
  \alpha \, \big(1  - \frac{\alpha}{2}\big)\, r_n + \SmallOhP{r_n} \,.
\]
The right-hand side is maximal for $\alpha = 1$ and is then ${r_n}/{2} +
\SmallOhP{r_n}$.
%]]]

\subsection{Upper bound on the number of connected components}
%[[[
Our strategy to get an upper bound on the number of connected components
is to find a spanning subgraph whose number of connected components we
can estimate. A natural idea is to look for a spanning forest, because
forests have the property that their number of connected components is their
number of vertices minus their number of edges.

\begin{definition} \label{defFoundingPairs}
A pair of vertices $\{ij\}$ is said to be a \emph{founder} if it has no
ancestor other than itself, i.e., letting
$T_{\{ij\}} = \sup \Set{t \geq 0 \suchthat a_t(i) \neq a_t(j)}$ be the
coalescence time of $i$ and $j$, $\{ij\}$ is a founder if and only if
$\forall t < T_{\{ij\}}$, $\{a_t(i)\,a_t(j)\} = \{ij\}$.
\end{definition}

Let $\mathpzc{F}$ be the set of founders of $G_{n,r_n} = (V, E)$, and let
$T_n = (V, \mathpzc{F})$. Note that $\#\mathpzc{F} = n - 1$ and that 
$T_n$ is a tree. Therefore, letting $F_n = (V, \mathpzc{F} \cap E)$ be the 
spanning forest of $G_{n,r_n}$ induced by $T_n$, we have
\[
  \nbcc_n \leq n - \#\mathrm{edges}(F_n) \, .
\]
Let us estimate the number of edges of $F_n$.
Recall Proposition~\ref{propBackwardRepresentation}.
By construction, $\forall \{ij\} \in \mathpzc{F}$,
$P^\star_{\{ij\}} = P_{\{ij\}} \cap [0, T_{\{ij\}}]$. It follows that
\[
  \#\mathrm{edges}(F_n) =
  \sum_{\{ij\} \in \mathpzc{F}}
  \Indic{P_{\{ij\}} \cap [0, T_{\{ij\}}] = \emptyset}
\]
and, as a consequence,
\[
  \nbcc_n \leq 1 + \sum_{\{ij\} \in \mathpzc{F}}
  \Indic{P_{\{ij\}} \cap [0, T_{\{ij\}}] \neq \emptyset} \, .
\]
Now, $\Indic{P_{\{ij\}} \cap [0, T_{\{ij\}}] \neq \emptyset} \leq
\#(P_{\{ij\}} \cap [0, T_{\{ij\}}])$, and since
$(P_{\{ij\}})_{\{ij\} \in \mathpzc{F}}$ are i.i.d.\ Poisson point processes
with intensity $r_n$ that are also independent of
$(T_{\{ij\}})_{\{ij\} \in \mathpzc{F}}$,
\[
  \sum_{\{ij\} \in \mathpzc{F}}\#(P_{\{ij\}} \cap [0, T_{\{ij\}}])
  \leq \#(P \cap [0, L_n]) \,, 
\]
where $P$ is a Poisson point process on $\OCInterval{0, +\infty}$ with
intensity $r_n$ and
$L_n = T_{\mathrm{MRCA}} + \sum_{\{ij\} \in \mathpzc{F}} T_{\{ij\}}$
is the total branch length of the genealogy of the vertices. Putting the
pieces together,
\[
  \nbcc_n \leq 1 + \#(P \cap [0, L_n])\,.
\]
Conditional on $L_n$, $\#(P \cap [0, L_n])$ is a Poisson random variable with
parameter $r_n L_n$. Moreover, it is known \cite{TavareTPB1984} that
\[
  \Expec{L_n} = 2 \sum_{i = 1}^{n - 1} \frac{1}{i}
  \quad\text{and}\quad
  \Var{L_n} = 4 \sum_{i = 1}^{n - 1} \frac{1}{i^2}
\]
As a result,
\[
  \Expec{\#(P \cap [0, L_n])} = r_n \Expec{L_n} \sim 2\,r_n \log n
\]
and
\[
  \Var{\#(P \cap [0, L_n])} = r_n \Expec{L_n} + \Var{r_n L_n}
  \sim 2\,r_n \log n + \alpha\,r_n^2 \, , 
\]
with $\alpha = 2\pi^2/3$. Using Chebychev's inequality, we find that
for all $\epsilon > 0$,
\[
  \Prob*{\big}{\Abs{\#(P \cap [0, L_n]) - 2\,r_n \log n} \geq \epsilon\,r_n \log n}
  = O\mleft( \frac{2}{\epsilon^2\,r_n \log(n)} +
    \frac{\alpha}{\epsilon^2\log(n)^2} \mright).
\]
The right-hand side goes to $0$ as $n \to +\infty$, which shows that
$\IsSmallOhP{\#(P \cap [0, L_n]) - 2\,r_n \log n}{r_n \log n}$ and finishes
the proof.

\begin{remark}
Using $\#(P \cap [0, L_n])$ as an upper bound for $\sum_{\{ij\} \in \mathpzc{F}}
\Indic{P_{\{ij\}} \cap [0, T_{\{ij\}}] \neq \emptyset}$ turns out not to be
a great source of imprecision, because most of the total branch length of a
Kingman coalescent comes from very short branches. As a result, when
$\IsSmallOh{r_n}{n}$, only a negligible proportion of the
$P_{\{ij\}} \cap [0, T_{\{ij\}}]$'s, $\{ij\} \in \mathpzc{F}$,
have more than one point.

By contrast, using $n - \#\mathrm{edges}(F_n)$ as an upper bound on
$\nbcc_n$ is very crude. This leads us to formulate the following conjecture:
\end{remark}
\begin{conjecture} \label{conjConnectedComponents}
\[
  \exists \alpha, \beta > 0 \st
  \Prob{\alpha r_n \leq \nbcc_n \leq \beta r_n} \tendsto{n \to \infty} 1.
\]
\end{conjecture}

%Soit $k$ à la fois $\omega(1)$ et $o(n)$. Alors quand on choisit une branche
%uniformément, avec proba qui tend vers 1 elle est née après la $k$-ième
%branche. Donc sa longueur est inférieure à $\sum_{i = k}^n 2/i^2 \leq (n - k) /
%k^2 \sim n / k^2$. Du coup, le nombre de mutations qu'elle porte est inférieur
%à $r_n n / k^2$. Si $r_n = \epsilon n$, avec $\epsilon = o(n)$, on prend $k =
%\omega(\sqrt{\epsilon} n)$ et alors le nombre de mutations d'une branche
%choisie au hasard est $o(1)$. Donc la majorité des branches ne portent pas de
%mutations. Le fait que le nombre moyen de mutation par branche soit $2r_n
%\log(n) / n$ --- qui peut ne pas tendre vers 0, voire tendre vers l'infini ---
%s'explique par le fait qu'il y une fraction négligeable de branches qui portent
%un nombre infini de mutations.
%]]]

%]]]

\section{Number of edges in the sparse regime} \label{secNumberOfEdges}
%[[[

%[[[
From the expressions obtained in section~\ref{secMomentsDegreeEdges} and
recapitulated in Table~\ref{tabFirstMoments},
we see that when $\IsSmallOmega{r_n}{n}$,
\[
  \Expec{|E_n|} \sim \Var{|E_n|} \sim \frac{n^2}{2\,r_n} \, .
\]
This suggests that the number of edges is Poissonian in the sparse regime,
and this is what the next theorem states.

\begin{theorem} \label{thmPoissonianEdges}
Let $|E_n|$ be the number of edges of $G_{n, r_n}$. If $\IsSmallOmega{r_n}{n}$
then
\[
  d_{\mathrm{TV}}\big(|E_n|, \mathrm{Poisson}(\lambda_n)\big)
  \tendsto{n \to +\infty} 0 \, , 
\]
where $d_{\mathrm{TV}}$ stands for the total variation distance and
$\lambda_n = \Expec{|E_n|} \sim \frac{n^2}{2 r_n}$. If in addition
$\IsSmallOh{r_n}{n^2}$, then $\lambda_n \to +\infty$ and as a result
\[
 \frac{|E_n| - \lambda_n}{\sqrt{\lambda_n}}
 \tendsto[\mathcal{D}]{n\to+\infty} \mathcal{N}(0, 1) \, ,
\]
where $\mathcal{N}(0, 1)$ denotes the standard normal distribution.
\end{theorem}

The proof of Theorem~\ref{thmPoissonianEdges} is a standard application of the
Stein--Chen method \cite{Stein1972, ChenAnnalsProb1975}.
A reference on the topic is \cite{Barbour1992},
and another excellent survey is given in \cite{RossProbabilitySurveys2011}.
Let us state briefly the results that we will need.

\begin{quoteddefinition} \label{defPositiveRelation}
The Bernoulli variables $X_1, \ldots, X_N$ are said to be \emph{positively
related} if for each $i = 1, \ldots, N$ there exists 
$(X_1^{(i)}, \ldots, X_N^{(i)})$, built on the same space as
$(X_1, \ldots, X_N)$, such that
\begin{mathlist}
\item $\big(X_1^{(i)}, \ldots, X_N^{(i)}\big) \sim (X_1, \ldots, X_N) \mid X_i = 1$.
\item For all $j = 1, \ldots N$, $X^{(i)}_j \geq X_j$.
\end{mathlist}
\end{quoteddefinition}

Note that there are other equivalent definitions of positive relation (see
e.g.\ Lemma~{4.27} in~\cite{RossProbabilitySurveys2011}).
Finally, we will need the following classic theorem, which
appears, e.g., as Theorem~{4.20} in~\cite{RossProbabilitySurveys2011}.

\begin{quotedtheorem} \label{thmSteinChenMethod}
Let $X_1, \ldots, X_N$ be positively related Bernoulli variables
with $\Prob{X_i = 1} = p_i$. Let $W = \sum_{i=1}^N X_i$ and
$\lambda = \Expec{W}$. Then, 
\[
  d_{\mathrm{TV}}(W, \mathrm{Poisson}(\lambda)) \leq
  \min\{1, \lambda^{-1}\}\left(\Var{W} - \lambda + 2 \sum_{i=1}^N p_i^2 \right).
\]
\end{quotedtheorem}
%]]]

\subsection{Proof of the positive relation between the edges}
%[[[

It is intuitive that the variables indicating the
presence of edges in our graph are positively related, because the only way
through which these variables depend on each other is through the fact that the
edges share ancestors. Our proof is nevertheless technical.

\subsubsection{Preliminary lemmas}
%[[[
In this section we isolate the proof of two useful results that are not tied
to the particular setting of our model.

\begin{lemma} \label{lemmaCharactVectorBernoulli}
Let $\mathbf{X} = (X_1, \ldots, X_N)$ be a vector of Bernoulli variables. The
distribution of $\mathbf{X}$ is uniquely characterized by the quantities
\[
  \Expec{\prod_{i \in I} X_i}, \quad
  I \subset \{1, \ldots, N\}, \, I \neq \emptyset
\]
\end{lemma}

\begin{proof}
For all $I \subset \{1, \ldots, N\}$, $I \neq \emptyset$, let
\[
  p_I = \Expec{\prod_{i \in I} X_i} \quad\text{and}\quad
  q_I = \Expec{\prod_{i \in I} X_i \prod_{j \in \Complement*{I}} (1 - X_j)} \,
\]
where the empty product is understood to be 1.

Clearly, the distribution of $\mathbf{X}$ is fully specified by $(q_I)$.
Now observe that, by the inclusion-exclusion principle,
\[
  q_I = \sum_{J \supset I} (-1)^{|J| - |I|} \, p_J \,,
\]
which terminates the proof.
\end{proof}

\begin{lemma} \label{lemmaPositiveRelation}
Let $X_1, \ldots, X_N$ be independent random nondecreasing functions from
$\COInterval{0, +\infty}$ to $\Set{0, 1}$ such that
\[
  \forall i \in \Set{1, \ldots, N},\quad
  \inf \Set{t \geq 0 \suchthat X_i(t) = 1} < +\infty \text{ almost surely.}
\]
Let $T$ be a non-negative random variable that is independent of 
$(X_1, \ldots, X_N)$. Then, $X_1(T), \ldots, X_N(T)$ are positively related.
\end{lemma}

\begin{proof}
Pick $i \in \Set{1, \ldots, N}$. Now, let
$\tau_i = \inf \Set{t \geq 0 \suchthat X_i(t) = 1}$.
Assume without loss of generality that $X_i$ is left-continuous, so that
$\Set{X_i(T) = 1} = \Set{T > \tau_i}$. Next, note that,
\[
  \forall x, t \geq 0, \quad
  \Prob{T > x, T > t} \geq \Prob{T > x} \Prob{T > t}\,.
\]
Integrating in $t$ against the law of $\tau_i$, we find that
\[
  \forall x \geq 0, \quad
  \Prob{T > x \given T > \tau_i} \geq \Prob{T > x}\,.
\]
This shows that $T$ is stochastically dominated by $T^{(i)}$, where
$T^{(i)}$ has the law of $T$ conditioned on $\Set{T > \tau_i}$.
As a result, there exists $S$, built on the same space as $X_1, \ldots, X_N$ and
independent of $(X_j)_{j \neq i}$, such that $S \sim T^{(i)}$ and
${S \geq T}$. For all $j \neq i$, let $X^{(i)}_j = X_j(S)$. Since $X_j$ is
nondecreasing, ${X^{(i)}_j \geq X_j(T)}$, and since $(X_j)_{j \neq i} \independent
(T, \tau_i)$, $(X^{(i)}_j)_{j \neq i} \sim {((X_j(T))_{j \neq i} \mid X_i(T) = 1)}$.
This shows that $X_1(T), \ldots, X_N(T)$ are positively related.
\end{proof}

\begin{remark} \label{remarkPositiveRelation}
Lemma~\ref{lemmaPositiveRelation} and its proof are easily adapted to the case
where $X_1, \ldots, X_N$ are nonincreasing and such that
$\inf \Set{t \geq 0 \suchthat X_i(t) = 0} < +\infty$ almost surely.
\end{remark}
%]]]

\subsubsection{Stein--Chen coupling}
%[[[

\begin{proposition} \label{propPositiveAssociation}
For any $n \geq 2$ and $r > 0$, the random variables $\Indic{i \lkto j}$ for
$\{ij\} \in V^{(2)}$, which indicate the presence of edges in $G_{n, r}$, are
positively related.
\end{proposition}

\begin{proof}
We use the forward construction described in
Section~\ref{secForwardConstruction} and proceed by induction.
To keep the notation light, throughout the rest of the proof we index the
pairs of vertices of $G^\star_r(n) = (\Set{1, \ldots, n}, E^\star_n)$ by the
integers from $1$ to $N = \binom{n}{2}$ and, for $i \in \Set{1, \ldots, N}$, we
let $X_i = \Indic{i \in E^\star_n}$. We also make consistent use of bold
letters to denote vectors, i.e., given any family of random variables
$Z_1, \ldots, Z_p$, we write $\mathbf{Z}$ for $(Z_1, \ldots, Z_p)$.

For $n = 2$, the family $X_i$ for $i \in \Set{1, \ldots, n}$ consists of a
single variable $X_1$, so it is trivially positively related.  

Now assume that $X_1, \ldots, X_N$ are positively related in $G^\star_r(n)$,
i.e.
\begin{align} \label{eqSChInduction}
  \forall i \leq N , & \;
  \exists \mathbf{Y}^{(i)} = \big(Y^{(i)}_1, \ldots, Y^{(i)}_N\big)
  \text{ such that } \nonumber \\
  \quad(\mathrm{i})&\quad \mathbf{Y}^{(i)} \sim (\mathbf{X} \mid X_i = 1) \\
  \quad(\mathrm{ii})&\quad \forall k \leq N, \; Y^{(i)}_k \geq X_k \nonumber 
\end{align}
Remember that $G^\star_r(n + 1)$ is obtained by (1) adding a vertex to
$G^\star_r(n)$ (which, without loss of generality, we label $n + 1$) and
linking it to a uniformly chosen vertex~$u_n$ of $G^\star_r(n)$ as well as
to the neighbors of $u_n$; and (2) waiting for an exponential time $T$
with parameter $\binom{n}{2}$ while removing each edge at constant rate~$r$.

Formally, $\forall k \leq N + n$, define the ``mother'' of $k$, $M_k \in
\Set{1, \ldots, N} \cup \{\emptyset\}$, by
\begin{itemize}
  \item If $k \leq N$ (i.e., if $k$ is the label of $\{u, v\}$, with
    $1 \leq u < v \leq n$), then $M_k = k$.
  \item If $k > N$ is the label of $\{v, n+1\}$, with $1 \leq v \leq n$,
    then $M_k = \ell$, where $\ell$ is the label of $\{u_n, v\}$.
  \item If $k > N$ is the label of $\{u_n, n + 1\}$,
    then $M_k = \emptyset$.
\end{itemize}

Letting $X'_k = \Indic{k \in E^\star_{n + 1}}$, we then have
\[
  X'_k = \begin{cases}
    \ {A_k}  &\text{if } M_k = \emptyset \\
    X_{M_k} {A_k} &\text{otherwise}  \\
  \end{cases} 
\]
with $A_k = \Indic{e_k > T}$, where we recall that $T \sim \mathrm{Exp}(N)$ and,
$e_1, \ldots, e_{N + n}$ are i.i.d.\ exponential variables with parameter~$r$
that are also independent of everything else.

Note that the random functions
$\widetilde{A}_k \colon t \mapsto \Indic{e_k > t}$,
$k \in \Set{1, \ldots N + n}$ are nonincreasing and such that
$\inf \Set*{t \geq 0 \suchthat \widetilde{A}_k(t) = 0} < +\infty$ almost
surely. By Lemma~\ref{lemmaPositiveRelation} (see also
Remark~\ref{remarkPositiveRelation}), it follows that $A_1, \ldots, A_{N + n}$
are positively related.

We now pick any $i \leq \binom{n + 1}{2} = N + n$ and build a vector
$\mathbf{Y}'^{(i)}$ that has the same law as $(\mathbf{X}'\mid X_i = 1)$ and
satisfies $\mathbf{Y}'^{(i)}\geq \mathbf{X}'$.

Assume that $M_i \neq \emptyset$. In that case,
\begin{enumerate}
  \item By the induction hypothesis, there exists $\mathbf{Y}^{(M_i)}$ that
    satisfies~\eqref{eqSChInduction}.
  \item Since by $A_1, \ldots, A_{N + n}$ are positively related,
    $\exists \mathbf{B}^{(i)} \sim (\mathbf{A}\mid A_i = 1)$ such that
    $\mathbf{B}^{(i)} \geq \mathbf{A}$. 
\end{enumerate}
Note that $\mathbf{A}$, $\mathbf{B}^{(i)}$, $\mathbf{X}$ and
$\mathbf{Y}^{(M_i)}$ are all built on the same space.
Therefore, omitting the $(M_i)$ and $(i)$ superscripts to keep the notation
light, we can set $Y'_i = 1$ and, for $k \neq i$,
\[
  Y'_k = \begin{cases}
    \ {B_k}  &\text{if } M_k = \emptyset\\
    Y_{M_k} {B_k} &\text{otherwise.} \\
  \end{cases} 
\]
With this definition, $\forall J \subset \{1, \ldots, N + n\}$, 
\[
  \Expec{\prod_{\,j \in J}\! Y'_j} =
  \Expec{\prod_{\,j \in \widetilde{J}}\! Y_j} \,
  \Expec{\prod_{\,j \in J}\! {B_j}} \, , 
\]
where $\widetilde{J} = \Set{M_j \suchthat j \in J, M_j \neq \emptyset}$.
By hypothesis,
\[
  \Expec{\prod_{\,j \in \widetilde{J}}\! Y_j} =
  \Expec{\prod_{\,j \in \widetilde{J}}\! X_j \given X_{M_i} = 1} =
  \Expec*{\Big}{X_{M_i} \prod_{\,j \in \widetilde{J}}\! X_j }
  \,\big/ \; \Expec*{\big}{X_{M_i}}
\]
Similarly,
\[
  \Expec{\prod_{\,j \in J}\! {B_j} } =
  \Expec{A_i \prod_{j \in J} {A_j} } \,\big/\; \Expec*{\big}{A_i}  \, .
\]
As a result,
\begin{align*}
  \Expec{\prod_{\,j \in J}\! Y'_j}
  &= \frac{\Expec{X_{M_i} \prod_{j \in \widetilde{J}} X_j}
  \Expec{{A_i} \prod_{j \in J} {A_j}}}{\Expec{X_{M_i}} \Expec{A_i}} \\
  &= \frac{\Expec{X_{M_i} A_i \prod_{j \in J} X'_j}}{\Expec{X_{M_i} A_i}} \\
  &= \Expec{\prod_{\,j \in J}\! X'_j \given X'_i = 1}
\end{align*}
By Lemma~\ref{lemmaCharactVectorBernoulli}, this shows that
$\mathbf{Y}' \sim (\mathbf{X}' \mid X'_i = 1)$.

If $M_i = \emptyset$, we can no longer choose $\mathbf{Y}^{(M_i)}$. 
However, in that case, $X'_i$ depends only on $A_i$.
Therefore, we set $Y'_i = 1$ and, for $k \neq i$,
\[
  Y'_k = X_{M_k} {B_k}
\]
Remembering that $X'_i = {A_i}$, we then check that
\[
  \Expec{\prod_{\,j\in J}\! Y'_j} =
  \frac{\Expec{\prod_{j\in \widetilde{J}} X_j}
  \Expec{A_i \prod_{j\in J} A_j}}{\Expec{A_i}}
  = \Expec{\prod_{\,j \in J}\! X'_j \given X'_i = 1}\, .
\]

Finally, it is clear that, with both constructions of
$\mathbf{Y}^{(i)}$, ${\mathbf{Y}'}^{(i)}_k \geq \mathbf{X}'_k$.
\end{proof}
%]]]

%]]]

\subsection{Proof of Theorem~\ref{thmPoissonianEdges}}
%[[[
Applying Theorem~\ref{thmSteinChenMethod}
to $|E_n| = \sum_{\{ij\}} \Indic{i \lkto j}$ and using the expressions in
Table~\ref{tabFirstMoments}, we get
\[
  d_{\mathrm{TV}}\Big(|E_n|, 
  \mathrm{Poisson}(\lambda_n)\Big) \leq \min \Set{1, \lambda_n^{-1}}\,C_n \, , 
\]
with $\lambda_n = \frac{n (n - 1)}{2 (r_n + 1)}$ and
\[
  C_n = \frac{n(n-1)(n^2r_n + 2nr_n^2 + nr_n - 2r_n^2 + 3r_n + 9)}{
  2\,(2\, r_n + 3) (r_n + 3) (r_n + 1)^2} \, .
\]
When $r_n = \omega(n)$,
\[
  \IsBigTheta{C_n}{\frac{n^4}{r_n^3} + \frac{n^3}{r_n^2}}\,.
\]
Now, if $r_n > \frac{n (n - 1)}{2} - 1$, so that
$\min \Set*{1, \lambda_n^{-1}} = 1$, we see that $\IsBigTheta{C_n}{n^3/r_n^2}$.
If by contrast $r_n \leq \frac{n (n - 1)}{2} - 1$ then
$\IsBigTheta{\lambda_n^{-1}C_n}{n/r_n}$. In both cases,
$\min\Set*{1, \lambda_n^{-1}}\,C_n$ goes to zero as $n \to +\infty$, proving the
first part of Theorem~\ref{thmPoissonianEdges}.

The convergence of $\frac{|E_n| - \lambda_n}{\sqrt{\lambda_n}}$ to the standard
normal distribution is a classic consequence of the conjunction of
$d_{\mathrm{TV}}(|E_n|, \mathrm{Poisson}(\lambda_n)) \to 0$
with $\lambda_n \to +\infty$. See, e.g.,~\cite{Barbour1992},
page~17, where this is recovered as a consequence of inequality~(1.39).
%]]]

%]]]

\section*{Acknowledgements}
%[[[
François Bienvenu thanks Jean-Jil Duchamps for helpful discussions.
All authors thank the \textit{Center for Interdisciplinary Research in Biology}
(CIRB) for funding.
Florence Débarre thanks the Agence Nationale de la Recherche for funding (grant
ANR-14-ACHN- 0003-01).
%]]]

%\section*{References}

\bibliographystyle{elsarticle-num}
\bibliography{biblio}

\begin{thebibliography}{10}
\expandafter\ifx\csname url\endcsname\relax
  \def\url#1{\texttt{#1}}\fi
\expandafter\ifx\csname urlprefix\endcsname\relax\def\urlprefix{URL }\fi
\expandafter\ifx\csname href\endcsname\relax
  \def\href#1#2{#2} \def\path#1{#1}\fi

\bibitem{ChungJCB2003}
F.~Chung, L.~Lu, T.~G. Dewey, D.~J. Galas, Duplication models for biological
  networks, J.\ Comput.\ Biol.\ 10~(5) (2003) 677--687.

\bibitem{IspolatovPhysRevE2005}
I.~Ispolatov, P.~Krapivsky, A.~Yuryev, Duplication-divergence model of protein
  interaction network, Phys.\ Rev.~E 71~(6) (2005) 061911.

\bibitem{SoleACS2002}
R.~Sol{\'{e}}, R.~Pastor-Satorras, E.~Smith, T.~B. Kepler, A model of
  large-scale proteome evolution, Adv.\ Complex Syst.\ 5~(1) (2002) 43--54.

\bibitem{VazquezComPlexUs2003}
A.~V{\'{a}}zquez, A.~Flammini, A.~Maritan, A.~Vespignani, Modeling of protein
  interaction networks, ComPlexUs 1 (2003) 38--44.

\bibitem{Poulton1904}
E.~B. Poulton, What is a species?, Proc.\ Entomol.\ Soc.\ Lond.\ 1903 (1904)
  lxxvii--cxvi.

\bibitem{Mayr1942}
E.~Mayr, Systematics and the Origin of Species from the Viewpoint of a
  Zoologist, Columbia University Press, 1942.

\bibitem{CoyneOrr2004}
J.~A. Coyne, H.~A. Orr, Speciation, Sinauer Associates, 2004.

\bibitem{Moran1958}
P.~A.~P. Moran, Random processes in genetics, Math.\ Proc.\ Camb.\ Philos.\
  Soc.\ 54~(1) (1958) 60--71.

\bibitem{Kingman1982}
J.~F.~C. Kingman, The coalescent, Stoch.\ Process.\ Appl.\ 13~(3) (1982)
  235--248.

\bibitem{Durrett2008}
R.~Durrett, Probability models for DNA sequence evolution, 2nd Edition,
  Springer-Verlag New York, 2008.

\bibitem{Etheridge2011}
A.~Etheridge, Some mathematical models from population genetics. {\'{E}}cole
  d{'}{\'{e}}t{\'{e}} de probabilit{\'{e}}s de {S}aint-{F}lour XXXIX-2009, Vol.
  2012, Springer Science \& Business Media, 2011.

\bibitem{Durrett2010}
R.~Durrett, Probability: theory and examples, 4th Edition, Cambridge University
  Press, 2010.

\bibitem{ChauvinRouaultSPA1991}
B.~Chauvin, A.~Rouault, A.~Wakolbinger, Growing conditioned trees, Stoch.\
  Process.\ Appl.\ 39~(1) (1991) 117--130.

\bibitem{LyonsPemantlePerezAnnProbab1995}
R.~Lyons, R.~Pemantle, Y.~Peres, Conceptual proofs of {$L \log L$} criteria for
  mean behavior of branching processes, Ann.\ Probab.\ 23~(3) (1995)
  1125--1138.

\bibitem{ManceauLambert2016}
M.~Manceau, A.~Lambert, The species problem from the modeler's point of view,
  bioRxiv~075580~\href {http://dx.doi.org/10.1101/075580}
  {\path{doi:10.1101/075580}}.

\bibitem{TavareTPB1984}
S.~Tavaré, Line-of-descent and genealogical processes, and their applications
  in population genetics models, Theor.\ Popul.\ Biol.\ 26~(2) (1984) 119--164.

\bibitem{Stein1972}
C.~M. Stein, A bound for the error in the normal approximation to the
  distribution of a sum of dependent random variables, in: L.~M. Le~Cam,
  J.~Neyman, E.~L. Scott (Eds.), Proc.\ Sixth Berkeley Symp.\ Math.\ Stat.\
  Probab.\, Vol.~2, University of California Press, 1972, pp. 583--602.

\bibitem{ChenAnnalsProb1975}
L.~H.~Y. Chen, Poisson approximation for dependent trials, Ann.\ Probab.\ 3~(3)
  (1975) 534--545.

\bibitem{Barbour1992}
A.~D. Barbour, L.~Holst, S.~Janson, Poisson approximation, Oxford Studies in
  Probability, Clarendon Press, 1992.

\bibitem{RossProbabilitySurveys2011}
N.~Ross, Fundamentals of {S}tein's method, Probab.\ Surv.\ 8 (2011) 201--293.

\bibitem{LambertBrazJProbabStat2017}
A.~Lambert, Probabilistic models for the (sub)tree(s) of life, Braz.\ J.\
  Probab.\ Stat.\ 31~(3) (2017) 415--475.

\end{thebibliography}

\newpage
\appendix

\section{Proofs of Propositions~\ref{propBackwardWithKingman}
and~\ref{propForwardConstruction} and of Lemma~\ref{lemmaBackwardConstruction}}
\label{appBackwardForward}
%[[[

\subsection{Proof of Propositions~\ref{propBackwardWithKingman} and~\ref{propForwardConstruction}}
%[[[
\begin{repproposition}{propBackwardWithKingman}
Let $(K_t)_{t \geq 0}$ be a Kingman coalescent on $V = \Set{1, \ldots, n}$,
and let $\pi_t(i)$ denote the block containing~$i$ in the corresponding
partition at time~$t$. Let the associated genealogy of pairs be the set 
\[
  \mathcal{G} = \Set*[\Big]{\big(t,\, \{\pi_t(i)\, \pi_t(j)\}\big) \suchthat
    \{ij\} \in V^{(2)},\, t \in \COInterval{0, T_{\{ij\}}}} \,, 
\]
where $T_{\{ij\}} = \inf\Set{t \geq 0 \suchthat \pi_t(i) = \pi_t(j)}$. 
Denote by
\[
  L_{\{ij\}} = \Set*[\Big]{\big(t,\, \{\pi_t(i)\, \pi_t(j)\}\big) \suchthat
   t \in \COInterval{0, T_{\{ij\}}}}
\]
the lineage of $\{ij\}$ in this genealogy. Finally, let $P^\bullet$ be a
Poisson point process with constant intensity $r_n$ on $\mathcal{G}$ and let
$G = (V, E)$, where
\[
  E = \Set{\{ij\} \in V^{(2)} \suchthat P^\bullet \cap L_{\{ij\}} = \emptyset} \, .
\]
Then, $G \sim G_{n, r_n}$.
\end{repproposition}

\begin{proof}
Let $(a_t)_{t \geq 0}$ and $\mathpzc{P}^\star$ be as in
Proposition~\ref{propBackwardRepresentation}, and let 
\[
  \mathcal{G}^\star = \Set*[\Big]{\big(t,\, \{a_t(i)\, a_t(j)\}\big)
  \suchthat \{ij\} \in V^{(2)},\, t \in \COInterval{0, T^\star_{\{ij\}}}} \, , 
\]
where $T^\star_{\{ij\}} = \inf\Set{t \geq 0 \suchthat a_t(i) = a_t(j)}$. 
Being essentially a finite union of intervals, $\mathcal{G}^\star$ 
can be endowed with the Lebesgue measure.

As already suggested, conditional on $(a_t)_{t \geq 0}$, $\mathpzc{P}^\star$
can be seen as a Poisson point process $P^\star$ with constant intensity $r_n$ 
on $\mathcal{G}^\star$. More specifically,
\[
  P^\star = 
  \Set*[\Big]{\big(t,\, \{a_t(i)\, a_t(j)\}\big)
  \suchthat \{ij\} \in V^{(2)},\, t \in P^\star_{\{ij\}}} \, .
\]
With this formalism, writing
\[
  L^\star_{\{ij\}} = \Set*[\Big]{\big(t,\, \{a_t(i)\, a_t(j)\}\big) \suchthat
   t \in \COInterval{0, T^\star_{\{ij\}}}}
\]
for the lineage of $\{ij\}$ in this genealogy, we see that $P^\star_{\{ij\}}$
is isomorphic to $P^\star \cap L^\star_{\{ij\}}$. In particular,
\[
  P^\star_{\{ij\}} = \emptyset \iff P^\star \cap L^\star_{\{ij\}} = \emptyset
\]
Now let $(\bar{\pi}_t)_{t \geq 0}$ be defined by
\[
  \forall i \in V, \quad \bar{\pi}_t(i) =
  \Set{j \in V \suchthat a_t(j) = a_t(i)} \, .
\]
Then, $\psi \colon (t,\, \{a_t(i)\, a_t(j)\}) \mapsto
(t,\, \{\bar{\pi}_t(i)\, \bar{\pi}_t(j)\})$ is a measure-preserving
bijection from $\mathcal{G}^\star$ to $\psi(\mathcal{G}^\star)$. Therefore,
$\psi(P^\star)$ is a Poisson point process with constant intensity $r_n$ on
$\psi(\mathcal{G}^\star)$. Since $(\bar{\pi}_t)_{t \geq 0}$ has the same law as
$(\pi_t)_{t \geq 0}$ from the proposition, we conclude that
\[
  \big(\psi(\mathcal{G}^\star), \, \psi(P^\star)\big) \sim
  (\mathcal{G}, P^\bullet)
\]
which terminates the proof.
\end{proof}

\begin{repproposition}{propForwardConstruction}
For any $r > 0$, for any integer $n \geq 2$, 
\[
  \Phi_n(G^\star_r(n)) \sim G_{n, r} \, .
\]
\end{repproposition}

\begin{proof}
First, let us give a Poissonian construction of $(G^\dagger_r(t))_{t \geq 0}$.
The edge-removal events can be recovered from a collection
$\mathpzc{P}^\dagger = \big(P^\dagger_{\{ij\}}\big)_{\{ij\} \in V^{(2)}}$
of i.i.d.\ Poisson point processes with rate~$r$ on $\R$ such
that, if $t \in P_{\{ij\}}$ and there is an edge between $i$ and $j$ in
$G^\dagger_r(t-)$, it is removed at time~$t$. The duplication events induce a
genealogy on the vertices of $G^\star_r(n)$ that is independent of 
$\mathpzc{P}^\dagger$. Using a backward-time notation, let
$a^\dagger_t(i)$ denote the ancestor of $i$ at time~${(t_n - t)}$, i.e.\
$t$~time-units before we reach $G^\star_r(n)$. Observe that, by construction
of $G^\star_r(n)$,
\[
  \{ij\} \in G^\star_r(n) \iff \Set{t \geq 0 \suchthat
  t \in P^\dagger_{\{a^\dagger_t(i)\, a^\dagger_t(j)\}}} = \emptyset \, .
\]

Taking the relabeling of vertices into account, the genealogy on the vertices
of $G^\star_r(n)$ translates into a genealogy on the vertices of
$\Phi_n(G^\star_r(n))$, where the ancestor $\bar{a}_t$ function is given by
$\bar{a}_t = \Phi_n \circ a^\dagger_t \circ \Phi_n^{-1}$. To keep only the
relevant information about this genealogy, define
\[
  \bar{\pi}_t(i) = \Set{j \in V \suchthat \bar{a}_t(j) = \bar{a}_t(i)}
\]
and let
\[
  \bar{\mathcal{G}} =
    \Set*[\Big]{\big(t,\, \{\bar{\pi}_t(i)\, \bar{\pi}_t(j)\}\big)
      \suchthat \{ij\} \in V^{(2)},\, t \in \COInterval{0, T_{\{ij\}}}} \,, 
\]
where $T_{\{ij\}} = \inf\Set{t \geq 0 \suchthat \bar{\pi}_t(i) = \bar{\pi}_t(j)}$. 
As before, let us denote by
\[
  \bar{L}_{\{ij\}} =
    \Set*[\Big]{\big(t,\, \{\bar{\pi}_t(i)\, \bar{\pi}_t(j)\}\big) \suchthat
     t \in \COInterval{0, T_{\{ij\}}}}
\]
the lineage of $\{ij\}$ in this genealogy. Finally, define
\[
  \bar{P} =
   \Set*[\Big]{\big(t,\, \{\bar{\pi}_t(i)\, \bar{\pi}_t(j)\}\big)
   \suchthat \{ij\} \in V^{(2)},\, t \in
   P^\dagger_{\{a^\dagger_t(\Phi_n^{-1}(i))\, a^\dagger_t(\Phi_n^{-1}(j))\}}}\,.
\]
Then, conditional on $\bar{\mathcal{G}}$, $\bar{P}$ is a Poisson point process
with constant intensity $r_n$ on $\bar{\mathcal{G}}$. Moreover,
\begin{align*}
  \{ij\}  \in \Phi_n(G^\star_r(n))
  &\iff \{\Phi_n^{-1}(i)\, \Phi_n^{-1}(j)\} \in G^\star_r(n) \\
  &\iff \Set{t \geq 0 \suchthat t \in P^\dagger_{\{a^\dagger_t(\Phi_n^{-1}(i))\,
   a^\dagger_t(\Phi_n^{-1}(j))\}}} = \emptyset \\
  &\iff \bar{P} \cap \bar{L}_{\{ij\}} = \emptyset \, .
\end{align*}
Therefore, by Proposition~\ref{propBackwardWithKingman}, to conclude the proof
it is sufficient to show that $(\bar{\pi}_t)_{t \geq 0}$ has the same law as
the corresponding process for a Kingman coalescent. By construction, the time
to go from $k$ to $k - 1$ blocks in $(\bar{\pi}_t)_{t \geq 0}$ is an
exponential variable with parameter $\binom{k}{2}$ and thus it only remains to
prove that the tree encoded by $(\bar{\pi}_t)_{t \geq 0}$ has the same
topology as the Kingman coalescent tree. This follows directly from the
standard fact that the shape of a Yule tree with $n$ tips labeled uniformly at
random with the integers from $1$ to $n$ is the same as that of the shape of a
Kingman $n$-coalescent tree -- namely, the uniform law on the set of ranked
tree shapes with $n$ tips labeled by $\Set{1, \ldots, n}$ (see
e.g.~\cite{LambertBrazJProbabStat2017}).

Alternatively, we can finish the proof as follows: working in backward time,
for $i = 1, \ldots, n - 1$, consider the $i$-th coalescence and let $U_i$
denote the mother in the corresponding duplication in the construction of
$G^\star_r(n)$. Note that $U_i \sim \mathrm{Uniform}(\Set{1, \ldots, n - i})$,
and that the coalescing blocks are then the block that contains
$\Phi_n(U_i)$ and the block that contains $\Phi_n(n - i + 1)$. Let us
record the information about the $i$ first coalescences in the variable
$\Lambda_i$ defined by $\Lambda_0 = \emptyset$ and, for $i \geq 1$,
\[
  \Lambda_i = \big(\Phi_n(n - k + 1), \, \Phi_n(U_k)\big)_{k = 1, \ldots, i} \,.
\]
Thus, we have to show that, conditional on $\Lambda_{i - 1}$, the block
containing $\Phi_n(n - i + 1)$ and the block containing $\Phi_n(U_i)$ are
uniformly chosen. We proceed by induction. For $i = 1$, this is trivial.
Now, for $i > 1$, observe that, conditional on $\Lambda_{i - 1}$, the
restriction of $\Phi_n$ to
\[
  I_i = \Set{1, \ldots, n} \setminus \Set{\Phi_n(n), \ldots, \Phi_n(n - i)}
\]
is a uniform permutation on $I_i$.
As a result, $\{\Phi_n(n - i + 1), \, \Phi_n(U_i)\}$ is a uniformly chosen
pair of elements of $I_i$ (note that the fact that $U_i$ is uniformly
distributed on $\Set{1, \ldots, n - i}$ is not necessary for this, but is
needed to ensure that the restriction of $\Phi_n$ to $I_{i + 1}$ remains
uniform when conditioning on $\Lambda_{i}$ in the next step of the
induction). Since each block contains exactly one element of $I_i$, this
terminates the proof.
\end{proof}
%]]]

\subsection{Proof of Lemma~\ref{lemmaBackwardConstruction}}
%[[[
\begin{replemma}{lemmaBackwardConstruction}
Let $S$ be a subset of $V^{(2)}$. Conditional on the measure $\mathpzc{M}$, for
any interval $I \subset \COInterval{0, +\infty}$ such that
\begin{mathlist}
\item For all $\{ij\} \in S$, $\forall t \in I$, $a_t(i) \neq a_t(j)$.
\item For all $\{k\ell\} \neq \{ij\}$ in $S$, $\forall t \in I, \;
  \{a_t(i), a_t(j)\} \neq \{a_t(k), a_t(\ell)\}$,
\end{mathlist}
$P^\star_{\{ij\}} \cap I$, $\{ij\} \in S$, are independent Poisson point
processes with rate~$r_n$ on~$I$.

Moreover, for any disjoint intervals $I$ and $J$,
$(P^\star_{\{ij\}} \cap I)_{\{ij\} \in S}$ is independent of
$(P^\star_{\{ij\}} \cap J)_{\{ij\} \in S}$.
\end{replemma}

\begin{proof}
For all $t \geq 0$, define $S_t$ by
\[
  S_t = \Set{\{a_t(i), a_t(j)\} \suchthat \{ij\} \in S} \, .
\]
Set $t_0 = \inf I$ and let $t_1, \ldots, t_{m - 1}$ be the jump times of
$(S_t)_{t \geq 0}$ on $I$, i.e.\
\[
  t_p = \inf \Set{t > t_{p - 1} \suchthat S_t \neq S_{t_{p - 1}}}\,, \quad 
  p = 1, \ldots, m - 1.
\]
Finally, set $t_m = \sup I$ and, for $p = 0, \ldots, m - 1$, let
$I_p = \COInterval{t_p, t_{p + 1}}$ and $\tilde{a}_p = a_{t_p}$, so that
$(\tilde{a}_p)_{p \in \{0, \ldots, m\}}$ is the embedded chain of
$(a_{t})_{t \in I}$. With this notation, for all $\{ij\} \in S$,
\[
  P^\star_{\{ij\}} \cap I = \bigcup_{p = 0}^{m - 1}
  \big(P_{\{\tilde{a}_p(i),\, \tilde{a}_p(j)\}} \cap I_p \big) \, , 
\]
where for $p \neq q$, $I_p \cap I_q = \emptyset$, and
$P_{\{uv\}}$, $\{uv\} \in V^{(2)}$, are i.i.d.\ Poisson point processes
on $\COInterval{0, +\infty}$ with rate~$r_n$. By assumption, for all
$p = 0, \ldots, m - 1$, for all $\{ij\} \neq \{k\ell\}$ in $S$,
$\tilde{a}_p(i) \neq \tilde{a}_p(j)$, 
$\tilde{a}_p(k) \neq \tilde{a}_p(\ell)$ and
$\{\tilde{a}_p(i), \tilde{a}_p(j)\} \neq \{\tilde{a}_p(k), \tilde{a}_p(\ell)\}$.
This shows that 
$(P_{\{\tilde{a}_p(i), \tilde{a}_p(j)\}} \cap I_p)$, $\{ij\} \in S$ and
$p = 0, \ldots, m-1$, are i.i.d.\ Poisson point processes with rate $r_n$ on the
corresponding intervals, proving the first part of the lemma.

The second assertion is proved similarly. Adapting the previous notation to
work with two disjoint intervals $I$ and $J$, i.e.\ letting
$(\tilde{a}^I_p)_{p \in \{0, \ldots, m_I\}}$ be the embedded chain of
$(a_{t})_{t \in I}$ and $(\tilde{a}^J_p)_{p \in \{0, \ldots, m_J\}}$ that of
$(a_{t})_{t \in J}$, for all $\{ij\} \in S$ we write
\[
  P^\star_{\{ij\}} \cap I = 
  \bigcup_{p = 0}^{m_I - 1}
  \big(P_{\{\tilde{a}^I_p(i),\, \tilde{a}^I_p(j)\}} \cap I_p \big) \,, 
\]
and
\[
  P^\star_{\{ij\}} \cap J = 
  \bigcup_{p = 0}^{m_J - 1}
  \big(P_{\{\tilde{a}^J_p(i),\, \tilde{a}^J_p(j)\}} \cap J_p \big) \,.
\]
We conclude the proof by noting that the families
\[
  \big(P_{\{\tilde{a}^I_p(i),\, \tilde{a}^I_p(j)\}} \cap I_p \big)_{\{ij\} \in
    S,\, p \in \Set{0, \ldots, m_I}}
\]
and
\[
  \big(P_{\{\tilde{a}^J_p(i),\, \tilde{a}^J_p(j)\}} \cap J_p \big)_{\{ij\} \in
    S,\, p \in \Set{0, \ldots, m_J}}
\]
are independent, because the elements of these families are either deterministic
(if, e.g, $\widetilde{a}^I_p(i) =\widetilde{a}^I_p(j)$, in which case
$P_{\{\tilde{a}^I_p(i),\, \tilde{a}^I_p(j)\}} = \emptyset$) or Poisson
point processes on intervals that are disjoint from each of the intervals
involved in the definition of the other family.
\end{proof}
%]]]

%]]]

%]]]

\pagebreak

\section{Proofs of Proposition~\ref{propCovDisjointEdges} and Corollary~\ref{colVarEdges}}
\label{appFirstMoments}
%[[[

%[[[
\begin{repproposition}{propCovDisjointEdges}
Let $i$, $j$, $k$ and $\ell$ be four distinct vertices of $G_{n, r_n}$. We have
\[
  \Cov{\Indic{i \lkto j}, \Indic{k \lkto \ell}} = 
  \frac{2\,r_n}{(1 + r_n)^2 (3 + r_n) (3 + 2\,r_n)}
\]
\end{repproposition}

\begin{repcorollary}{colVarEdges}
Let $D^{(i)}_n$ and $D^{(j)}_n$ be the respective degrees of two fixed vertices
$i$ and $j$, and let $\Abs{E_n}$ be the number of edges of $G_{n, r_n}$.
We have
\[
  \Cov{D^{(i)}_n, D^{(j)}_n} =
  \frac{r_n}{(1 + r_n)^2} \mleft(1 + \frac{3 (n - 2)}{3 + 2\,r_n} +
  \frac{2 (n - 2)(n - 3)}{(3 + r_n)(3 + 2\,r_n)}\mright)
\]
and
\[
  \Var{\Abs{E_n}} = 
  \frac{r_n\, n\,(n - 1)(n^2 + 2\,r_n^2 + 2\,n\,r_n + n + 5\,r_n + 3)}{2 \,
    (1 + r_n)^2 \, (3 + r_n) \, (3 + 2\,r_n)}
\]
\end{repcorollary}
%]]]

\subsection{Proof of Proposition~\ref{propCovDisjointEdges}}
%[[[
The proof of Proposition~\ref{propCovDisjointEdges} parallels that of
Proposition~\ref{propCovOverlappingEdges}, but this time the topology of the
genealogy of the pairs of vertices has to be taken into account. Indeed,
define
\[
  S_t = \Set{a_t(i), a_t(j), a_t(k), a_t(\ell)}
\]
and let $\tau_1 < \tau_2 < \tau_3$ be the times of coalescence in the
genealogy of $\Set{i, j, k, \ell}$, i.e.
\[
  \tau_p = \inf\Set{t \geq 0\suchthat \Abs{S_t} = 4 - p}, \quad p = 1, 2, 3 \,.
\]
Write $I_1 = \COInterval{0, \tau_1}$, $I_2 = \COInterval{\tau_1,
\tau_2}$ and $I_3 = \COInterval{\tau_2, \tau_3}$. Finally, for $m = 1, 2$, let
\[
  A^{(m)}_{\{uv\}} = \Set{a_{\tau_m-}(u) \neq a_{\tau_m-}(v)} \cap
  \Set{a_{\tau_m}(u) = a_{\tau_m}(v)}
\]
be the event that the $m$-th coalescence in the genealogy of
$\Set{i, j, k, \ell}$ involved the lineages of $u$ and $v$ (note that the third
coalescence is uniquely determined by the first and the second, so we do not
need $A^{(3)}_{\{uv\}}$).

On $A^{(1)}_{\{ij\}} \cap A^{(2)}_{\{k\ell\}}$,
$\Set{i \lkto j, k \lkto \ell}$ is equivalent to 
\[
  (P^\star_{\{ij\}} \cap I_1) \cup
  (P^\star_{\{k\ell\}} \cap I_1) \cup
  (P^\star_{\{k\ell\}} \cap I_2)= \emptyset
\]
so that, conditionally on $I_1$ and $I_2$,
by Lemma~\ref{lemmaBackwardConstruction},
\begin{align*}
  \Prob{i \lkto j, k \lkto \ell \given A^{(1)}_{\{ij\}} \cap A^{(2)}_{\{k\ell\}}}
  &= \Prob{(P^\star_{\{ij\}} \cup P^\star_{\{k\ell\}}) \cap I_1 = \emptyset}
  \times
  \Prob{P^\star_{\{k\ell\}} \cap I_2 = \emptyset} \\
  &= \frac{6}{6 + 2\,r_n} \times \frac{3}{3 + r_n} \, .
\end{align*}
By contrast, on $A^{(1)}_{\{ij\}} \cap A^{(2)}_{\{ik\}}$,
$\Set{i \lkto j, k \lkto \ell}$ is
\[
  (P^\star_{\{ij\}} \cap I_1) \cup
  (P^\star_{\{k\ell\}} \cap I_1) \cup
  (P^\star_{\{k\ell\}} \cap I_2) \cup
  (P^\star_{\{k\ell\}} \cap I_3) = \emptyset
\]
and thus
\[
  \Prob{i \lkto j, k \lkto \ell \given A^{(1)}_{\{ij\}} \cap A^{(2)}_{\{ik\}}}
  = \frac{6}{6 + 2\,r_n} \times \frac{3}{3 + r_n} \times \frac{1}{1 + r_n}\, .
\]

Given a realization of the topology of the genealogy of the form
$A^{(1)}_{\{u_1 v_1\}} \cap A^{(2)}_{\{u_2 v_2\}}$, we can always express
$\Set{i \lkto j, k \lkto \ell}$ as a union of intersections of
$P^\star_{\{ij\}}$ and $P^\star_{\{k\ell\}}$ with $I_1$, $I_2$ and $I_3$.
In total, there are $\binom{4}{2} \times \binom{3}{2} = 18$ possible
events $A^{(1)}_{\{u_1 v_1\}} \cap A^{(2)}_{\{u_2 v_2\}}$, each having
probability $1/18$. This enables us to compute
$\Prob{i \lkto j, k \lkto \ell}$, but in fact the calculations can be
simplified by exploiting symmetries, such as the fact that $\{ij\}$ and
$\{k\ell\}$ are interchangeable.  In the end, it suffices to consider four
cases, as depicted in Figure~\ref{figAppCovDisjointEdges}.

\begin{figure}[h!]
  \centering
  \captionsetup{width=\linewidth}
  \includegraphics[width=\linewidth]{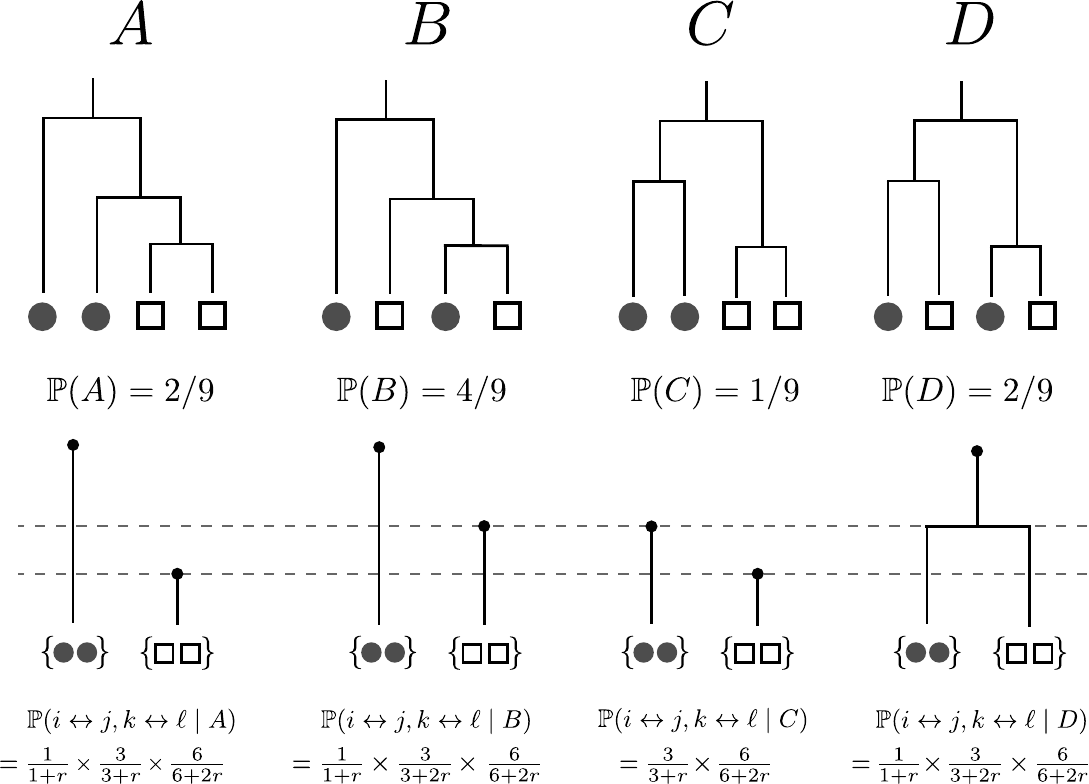}
  \caption
  {The four cases that we consider to compute
  $\Prob{i \lkto j, k \lkto \ell}$. Top, the ``aggregated'' genealogies of
  vertices and their probability. Each of these correspond to several
  genealogies on $\Set{i, j, k, \ell}$, which are obtained by labeling
  symbols in such a way that a pair of matching symbols has to correspond to
  either $\{ij\}$ or $\{k\ell\}$. For instance,
  $C = (A^{(1)}_{\{ij\}} \cap A^{(2)}_{\{k\ell\}}) \cup
  (A^{(1)}_{\{k\ell\}} \cap A^{(2)}_{\{ij\}})$ and therefore
  $\Prob{C} = 2 / 18$. Similarly, $A = (A^{(1)}_{\{ij\}} \cap A^{(2)}_{\{ik\}})
  \cup (A^{(1)}_{\{ij\}} \cap A^{(2)}_{\{i\ell\}}) \cup (A^{(1)}_{\{k\ell\}}
  \cap A^{(2)}_{\{ik\}}) \cup (A^{(1)}_{\{k\ell\}} \cap A^{(2)}_{\{jk\}})$ and
  $\Prob{A} = 4/18$, etc.
  Bottom, the associated genealogy of the pairs and
  the corresponding conditional probability of
  $\Set{i\lkto j, k \lkto \ell} \iff* \Set{\square \lkto \square,
  \text{\Large$\bullet$} \lkto \text{\Large$\bullet$}}$.}
  \label{figAppCovDisjointEdges}
\end{figure}

Putting the pieces together, we find that
\begin{align*}
  \Prob{i \lkto k, j \lkto \ell} &= \quad
  \frac{6}{9} \times \frac{1}{1 + r_n}
  \times \frac{3}{3 + 2r} \times \frac{6}{6 + 2\,r_n} \\
  &\quad + \frac{2}{9} \times \frac{1}{1 + r_n} \times \frac{3}{3 + r_n}
  \times \frac{6}{6 + 2\,r_n} \\
  &\quad + \frac{1}{9} \times \frac{3}{3 + r_n} \times \frac{6}{6 + 2\,r_n} \\[2ex]
  &= \frac{9 + 2\,r_n}{(1 + r_n)(3 + r_n)(3 + 2\,r_n)}\,.
\end{align*}
and Proposition~\ref{propCovDisjointEdges} follows, since
\[
  \Prob{i \lkto j}\, \Prob{k \lkto \ell} = \mleft(\frac{1}{1 + r_n}\mright)^2 \,.
\]
%]]]

\subsection{Proof of Corollary~\ref{colVarEdges}}
%[[[
Corollary~\ref{colVarEdges} is proved by standard calculations. First,
\begin{align*}
  \Cov{D^{(i)}_n, D^{(j)}_n} &= \Cov{\sum_{k\neq i} \Indic{i \lkto k},
  \sum_{\ell \neq j} \Indic{j \lkto \ell}} \\
  &= \Var{\Indic{i \lkto j}} \\
  &\quad+ 3(n - 2) \Cov{\Indic{i \lkto k}, \Indic{j \lkto k}} \\
  &\quad+ (n - 2) (n - 3) \Cov{\Indic{i \lkto k}, \Indic{j \lkto \ell}}
\end{align*}
Remembering from Proposition~\ref{propProbaEdge} that
$\Var[\normalsize]{\Indic{i \lkto j}} = r_n/(1 + r_n)^2$ and from
Proposition~\ref{propCovOverlappingEdges} that
$\Cov[\normalsize]{\Indic{i \lkto k}, \Indic{j \lkto k}} =
\frac{r_n}{(1 + r_n)^2(3 + 2\,r_n}$, and using
Proposition~\ref{propCovDisjointEdges}, we find that
\[
  \Cov{D^{(i)}_n, D^{(j)}_n} =
  \frac{r_n}{(1 + r_n)^2} \mleft(1 + \frac{3 (n - 2)}{3 + 2\,r_n} +
  \frac{2 (n - 2)(n - 3)}{(3 + r_n)(3 + 2\,r_n)}\mright) \, .
\]

Finally, to compute $\Var{\Abs{E_n}}$, we could do a similar calculation.
However, it is easier to note that
\[
\Abs{E_n} = \frac{1}{2} \sum_{i = 1}^n D^{(i)}_n \, .
\]
As a result, 
\begin{align*}
  \Var{\Abs{E_n}} &= \frac{1}{4} \mleft( n \Var{D^{(i)}_n} + n(n - 1)
  \Cov{D^{(i)}_n, D^{(j)}_n}\mright)  \\
  &= \frac{r_n\,  n\, (n - 1) (n^2 + 2\,r_n^2 + 2\,n\,r_n + n + 5\,r_n + 3)}{
  2 (1 + r_n)^2 (3 + r_n) (3 + 2\,r_n)} \, .
\end{align*}
%]]]

%]]]

\section{Proof of Theorem~\ref{thmDegreeConvergence}}
\label{appDegreeConvergence}
%[[[
In this section, we prove Theorem~\ref{thmDegreeConvergence}.

\begin{reptheorem}{thmDegreeConvergence}[convergence of the rescaled degree] ~
\begin{mathlist}
  \item If $r_n \to r > 0$, then $\frac{D_n}{n}$ converges in distribution to a 
    $\mathrm{Beta}(2, 2\,r)$ random variable.
  \item If $r_n$ is both $\SmallOmega{1}$ and $\SmallOh{n}$, then
    $\frac{D_n}{n / r_n}$ converges in distribution to a size-biased
    exponential variable with parameter~$2$.
  \item If $2\,r_n / n \to \rho > 0$, then $D_n + 1$ converges in distribution to a
    size-biased geometric variable with parameter $\rho/(1 + \rho)$.
\end{mathlist}
\end{reptheorem}

The proof of (iii) is immediate: indeed, by Theorem~\ref{thmDegreeDistribution},
\[
  \Prob{D_n + 1 = k} = \frac{2\,r_n\,(2\,r_n + 1)}{(n + 2\,r_n) (n - 1 + 2\,r_n)}
  \, k \, \prod_{i = 1}^{k - 1} \frac{n - i}{n - i + 2\,r_n - 1}  \, .
\]
If $2\,r_n / n \to \rho$, then for any fixed~$k$ this goes to
$k \mleft(\frac{\rho}{1+\rho}\mright)^2 \mleft(\frac{1}{1+\rho}\mright)^{k-1}$
as $n \to +\infty$.

\subsection{Outline of the proof}
%[[[
To prove (i) and (ii), we show the pointwise convergence of the cumulative
distribution function $F_n$ of the rescaled degree. To do so, in both cases,
\begin{enumerate}
  \item We show that, for any $\epsilon > 0$, for $n$ large enough,
    \[
      \forall y \geq 0, \quad \int_0^y f_n(x) \, dx
      \leq F_n(y) \leq \int_0^{y + \varepsilon} f_n(x) \, dx
    \]
    for some function $f_n$ to be introduced later.
  \item We identify the limit of $f_n$ as a classical probability density
    $f$, and use dominated convergence to conclude that
    \[
      \forall y \geq 0, \quad \int_0^y f_n(x) \, dx \to \int_0^y f(x) \, dx \,.
    \]
\end{enumerate}

In order to factorize as much of the reasoning as possible, we introduce
the rescaling factor $N_n$:
\begin{itemize}
  \item When $r_n \to r$, i.e.\ when we want to prove (i), $N_n = n$.
  \item When $r_n$ is both $\SmallOmega{1}$ and $\SmallOh{n}$,
    i.e.\ when we want to prove (ii), $N_n = n/r_n$.
\end{itemize}
Thus, in both cases the rescaled degree is $D_n / N_n$ and its cumulative
distribution function is
\[
  F_n(y) = \sum_{k = 0}^{\lfloor N_n y \rfloor} \Prob{D_n = k} \, .
\]
%]]]

\subsection{Step 1}
%[[[
For all $x > 0$, let
\[
  f_n(x) = N_n \Prob{D_n = \lfloor N_n x \rfloor}, \, 
\]
so that
\[
  \forall k \in \N, \quad
  \Prob{D_n = k} = \int_{k/N_n}^{(k + 1)/N_n} f_n(x)\,dx  \, .
\]
If follows that
\[
  F_n(y) = \int_0^{(\lfloor N_n y \rfloor + 1) / N_n} f_n(x) \, dx \, .
\]
Finally, since $y \leq \frac{\lfloor N_n y \rfloor + 1}{N_n} \leq y +
\frac{1}{N_n}$ and $f_n$ is non-negative, for any $\varepsilon > 0$, for $n$
large enough,
\[
   \forall y \geq 0, \quad
   \int_0^y f_n(x) \, dx \leq F_n(y) \leq \int_0^{y + \varepsilon} f_n(x)\,dx\,,
\]
and the rank after which these inequalities hold is uniform in $y$,
because the convergence of $(\lfloor N_n y \rfloor + 1)/{N_n}$ to $y$ is.
%]]]

\subsection{Step 2}
%[[[
To identify the limit of $f_n$, we reexpress it in terms of the gamma
function. Using that $\Gamma(z) = z\Gamma(z)$, by induction,
\[
  \prod_{i = 1}^k (n - i) = \frac{\Gamma(n)}{\Gamma(n - k)}
  \quad\text{and}\quad
  \prod_{i = 1}^k (n - i + 2\,r_n - 1) =
  \frac{\Gamma(n + 2\,r_n - 1)}{\Gamma(n - k + 2\,r_n - 1)} \, .
\]
Therefore, $f_n(x)$ can also be written
\begin{equation} \label{eqAppDegreeConv01}
  f_n(x) = 
  \frac{N_n\, 2\,r_n\, (2\,r_n + 1)}{(n + 2\,r_n) (n - 1 + 2\,r_n)} \,
  \big(\lfloor N_n x \rfloor + 1\big) \times P_n(x) \, , 
\end{equation}
where
\begin{equation} \label{eqAppDegreeConv02}
  P_n(x) =
  \frac{\Gamma(n)\,  \Gamma(n - \lfloor N_n x\rfloor + 2\,r_n - 1)}{\Gamma(n -
  \lfloor N_n x \rfloor)\, \Gamma(n + 2\,r_n - 1)} \, .
\end{equation}

We now turn to the specificities of the proofs of (i) and (ii).

\pagebreak

\subsubsection{Proof of (i)}
%[[[
In this subsection, $r_n \to r > 0$ and $N_n = n$.

\paragraph*{Limit of $f_n$}
%[[[
Recall that
\[
  \forall \alpha \in \R, \quad
  \frac{\Gamma(n + \alpha)}{\Gamma(n)} \sim n^\alpha \, .
\]
Using this in \eqref{eqAppDegreeConv02}, we see that, for all
$x \in \COInterval{0, 1}$, 
\[
  P_n(x) \to  (1 - x)^{2r - 1} \, .
\]
Therefore, for all $x \in \COInterval{0, 1}$,
\[
  f_n(x) \to 2r(2r + 1)\,  x\,  (1 - x)^{2r - 1} \, .
\]
Noting that $2r(2r + 1) = 1/ B(2, 2r)$, where $B$ denotes the beta
function, we can write $f = \lim_n f_n$ as
\[
  f\colon x \mapsto \frac{x(1 - x)^{2r - 1}}{B(2, 2r)} \,
                \Indic*{\COInterval{0, 1}}(x)
\]
and we recognize the probability density function of the $\text{Beta}(2, 2r)$
distribution.
%]]]

\paragraph*{Domination of $(f_n)$}
%[[[
First note that, for all $x \in \COInterval{0, 1}$,
\[
  \frac{1}{n - 1 + 2\,r_n} \prod_{i = 1}^{\lfloor n x \rfloor}
  \frac{n - i}{n - i + 2\,r_n - 1} \quad = \quad
  \frac{1}{n - \lfloor nx \rfloor + 2\,r_n - 1}
  \prod_{i = 1}^{\lfloor n x \rfloor}
  \frac{n - i}{n - i + 2\,r_n} \, , 
\]
where the empty product is understood to be $1$.  Since $2\,r_n > 0$, this
enables us to write that, for all $x \in \COInterval{0, 1}$,
\[
  f_n(x) = 
  \underbrace{\frac{n\, 2r\, (2r + 1)}{n + 2r}}_{\leq (2r + 1)^2} \times 
  \frac{\lfloor n x \rfloor + 1}{n - 1 + 2r} \times
  \underbrace{\frac{1}{n - \lfloor nx \rfloor + 2r - 1}}_{\leq \frac{1}{2r}}
  \times \underbrace{\prod_{i = 1}^{\lfloor n x \rfloor}
  \frac{n - i}{n - i + 2r}}_{\leq 1}\, .
\]
where, to avoid cluttering the expression, the $n$ index of $r_n$ has been
dropped. Since
\[
  \frac{\lfloor n x \rfloor + 1}{n - 1 + 2\,r_n} \leq
  \frac{(n - 1) x  + x + 1}{n - 1} \leq x + \frac{2}{n - 1} \, 
  \xrightarrow[n\to +\infty]{\,\,\text{uniformly}\,\,}\, x \, , 
\]
there exists $c$ such that, for all $x \in \COInterval{0, 1}$ and
$n$ large enough,
\[
  f_n(x) \leq c\, x
\]
Since $f_n$ is zero outside of $\COInterval{0, 1}$, this shows that
$(f_n)$ is dominated by $g\colon x \mapsto c \, x \, \mathds{1}_{[0, 1[}(x)$.
%]]]

%]]]

\subsubsection{Proof of (ii)}
%[[[
In this subsection, $r_n$ is both $\SmallOmega{1}$ and
$\SmallOh{n}$, and $N_n = n / r_n$.
For brevity, we will write $k_n$ for $\lfloor nx/r_n\rfloor$. It should be
noted that
\begin{itemize}
  \item $k_n$ is both $\SmallOmega{1}$ and $\SmallOh{n}$.
  \item $k_n r_n / n \to x$ uniformly in $x$ on $\COInterval{0, +\infty}$.
\end{itemize}

\paragraph*{Limit of $f_n$}
%[[[
In this paragraph, we will need Stirling's formula for the asymptotics of
$\Gamma$:
\[
  \Gamma(t + 1) \sim \sqrt{2\pi t} \, \frac{t^t}{e^t}\, .
\]
Using this in Equation~\eqref{eqAppDegreeConv02}, 
\begin{align*}
  P_n(x) &=
  \frac{\Gamma(n)\,  \Gamma(n - \lfloor N_n x\rfloor + 2\,r_n - 1)}{\Gamma(n -
  \lfloor N_n x \rfloor)\, \Gamma(n + 2\,r_n - 1)} \\[1ex]
  &\sim \underbrace{\sqrt{\frac{(n - 1)(n - 2 - k_n + 2\,r_n)}{
  (n - 1 - k_n)(n - 2 + 2\,r_n)}}}_{\sim 1} \times
  \underbrace{\frac{e^{n - 1 - k_n}\, e^{n - 2 + 2r_n}}{e^{n - 1}\, 
  e^{n - 2 - k_n + 2r_n}}}_{ = 1} \times \, Q_n
\end{align*}
where
\[
  Q_n = 
  \frac{(n - 1)^{n - 1} \, 
  (n - 2 - k_n + 2\,r_n)^{n - 2 - k_n + 2r_n}}{(n - 1 - k_n)^{n - 1 - k_n} \,
  (n - 2 + 2\,r_n)^{n - 2 + 2r_n}} \, .
\]
Let us show that $Q_n\to e^{-2x}$:
\begin{align*}
  \log Q_n &= \ \ (n - 1)\log({n - 1}) \\
         &\ \ + (n - a + b)\log({n - a + b}) \\
         &\ \ -(n - a)\log({n - a}) \\
         &\ \ - (n - 1 + b)\log({n - 1 + b})
\end{align*}
where, to avoid cluttering the text, we have written $a$ for $k_n + 1$ and
$b$ for $2\,r_n - 1$. Factorizing, we get
\[
  \log Q_n = n \log\!\left(\frac{(n - 1)(n - a + b)}{(n - a)(n - 1 + b)}\right)
  - a \log\!\left(\frac{n - a + b}{n - a} \right)
  + b \log\!\left(\frac{n - a + b}{n - 1 + b}\right)
  - \log\!\left(\frac{n - 1}{n - 1 + b}\right) \, .
\]
Now,
\[
  \frac{(n - 1)(n - a + b)}{(n - a)(n - 1 + b)} =
  1 + \underbrace{\frac{(a - 1) b}{n^2 -n + nb - na + a - ab}}_{
  \sim \frac{2 k_nr_n}{n^2}\; =\; o(1)}
\]
so that
\[
  n \log\!\left(\frac{(n - 1)(n - a + b)}{(n - a)(n - 1 + b)}\right)
  \sim \frac{2 k_n r_n}{n} \to 2x
\]
Similarly,
\begin{align*}
  - a \log\!\left(\frac{n - a + b}{n - a} \right) =
  - a \log\!\left(1 + \frac{b}{n - a} \right) \sim - \frac{ab}{n} \to -2x \\ \\
   b \log\!\left(\frac{n - a + b}{n - 1 + b}\right) =
   b \log\!\left(1 + \frac{1 - a}{n - 1 + b}\right) \sim - \frac{ab}{n} \to -2x \\ \\
\end{align*}
and, finally, $- \log\!\left(\frac{n - 1}{n - 1 + b}\right) \to 0$. Putting
the pieces together,
\[
  \log Q_n \to - 2x \, .
\]

Having done that, we note that
\[
  \frac{2 \, n (2\,r_n + 1)}{(n + 2\,r_n)(n - 1 + 2\,r_n)}(k_n + 1) \,\to\, 4\,x\,.
\]
Plugging these results in Equation~\eqref{eqAppDegreeConv01}, we see that
\[
  \forall x \in \R, \quad f_n(x) \;\to\; 4\,x\,e^{-2x} \,
  \Indic*{\COInterval{0, +\infty}}(x)
\]
and we recognize the probability density function of a size-biased exponential
distribution with parameter 2.
%]]]

\paragraph*{Domination of $(f_n)$}
%[[[
Recall that, since $N_n = n/r_n$, for all $x \in \COInterval{0, 1}$, 
\[
  f_n(x) = 
  \frac{2\,n\,(2\,r_n + 1)}{(n + 2\,r_n) (n - 1 + 2\,r_n)}
  \, (k_n + 1) \, \prod_{i = 1}^{k_n} \frac{n - i}{n - i + 2\,r_n - 1}  \, .
\]
Next, note that, for all $i$, 
\[
  \frac{n - i}{n - i + 2\,r_n - 1} = 1 - \frac{2\,r_n - 1}{n - i + 2\,r_n - 1}
  \leq \exp\left({-\frac{2\,r_n - 1}{n - i + 2\,r_n - 1}}\right)
\]
so that
\[
  \prod_{i = 1}^{k_n} \frac{n - i}{n - i + 2\,r_n - 1} \leq
  \exp\left(- \sum_{i = 1}^{k_n}{\frac{2\,r_n - 1}{n - i + 2\,r_n - 1}}\right) \, ,
\]
with
\[
  \sum_{i = 1}^{k_n}{\frac{2\,r_n - 1}{n - i + 2\,r_n - 1}} \geq k_n\, 
  \frac{2\,r_n - 1}{n - 1 + 2\,r_n - 1} \, .
\]
Because $r_n = \omega(1)$, for all $\varepsilon > 0$,
$2\,r_n - 1 \geq (1 - \varepsilon) 2\,r_n$ for $n$ large enough.
Similarly, since $r_n = o(n)$,
$\frac{1}{n + 2\,r_n} \geq \frac{1}{(1 + \varepsilon) n}$.
As a result, there exists $c > 0$ such that
\[
  k_n\, \frac{2\,r_n - 1}{n - 1 + 2\,r_n - 1} \geq c \, k_n\, \frac{2\,r_n}{n}
  \,\tendsto[\text{uniformly}]{}\, 2cx \, .
\]
We conclude that 
\[
  \forall x \geq 0, \quad 
  \prod_{i = 1}^{k_n} \frac{n - i}{n - i + 2\,r_n - 1}
  \leq \exp\mleft(-2cx\mright)
\]
for $n$ large enough. Finally,
\[
  2\times \underbrace{\frac{n}{n + 2\,r_n}}_{\leq 1} \times
  \underbrace{\frac{(2\,r_n + 1) (k_n + 1)}{(n - 1 + 2\,r_n)}}_{\to 2x,\, \text{uniformly}}
  \leq 4cx
\]
and so $(f_n)$ is dominated by
$g\colon x \mapsto 4 c\, x \, e^{-2cx} \, \Indic*{\COInterval{0, +\infty}}(x)$.
%]]]

%]]]

%]]]

%]]]

%\section*{R02 --- Calcul alternatif pour $\Expec{D_n}$}
%\label{R02}
%%[[[
%We proceed by induction on $n$. First, it is easy to check that the equality
%holds for $n = 2$. Now, $\forall n \geq 2$, let
%\[
%  \alpha_{n} \, = \,  \frac{2r ( 2r + 1)}{(n + 2r)(n - 1 + 2r)} =
%  \Prob{D_n \, = \,  0} \, .
%\]
%Then, $\Expec{D_{n + 1}}$ is
%\begin{align*}
%  &\quad \alpha_{n + 1} \sum_{k = 0}^n k(k + 1) \prod_{i = 1}^k
%    \frac{n + 1 - i}{n + 1 - i + 2r - 1} \\ \\
%  &= \alpha_{n + 1} \left(\frac{2\, n}{n + 2r - 1} + \sum_{k = 2}^{n} k (k + 1)
%    \prod_{i = 0}^{k - 1} \frac{n - i}{n - i + 2r - 1}\right) \\ \\
%  &= \alpha_{n + 1} \left(\frac{2\, n}{n + 2r - 1} +
%    \sum_{k = 1}^{n - 1} (k + 1) (k + 2)
%    \prod_{i = 0}^{k} \frac{n - i}{n - i + 2r - 1}\right) \\ \\
%  &= \frac{\alpha_{n + 1}\, n}{n + 2r - 1} \left(2 + \sum_{k = 1}^{n - 1}
%    k (k + 1) \prod_{i = 1}^{k} \frac{n - i}{n - i + 2r - 1} +
%  2 \sum_{k = 1}^{n - 1} (k + 1) \prod_{i = 1}^{k} \frac{n - i}{n - i + 2r - 1}\right) \\ \\
%  &= \frac{\, n}{n + 2r + 1} \left(2\, \alpha_{n} +
%  \underbrace{\alpha_{n}\, \sum_{k = 1}^{n - 1} k (k + 1) \prod_{i = 1}^{k}
%  \frac{n - i}{n - i + 2r - 1}}_{\Expec{D_n}} +
%  2 \underbrace{\alpha_{n} \sum_{k = 1}^{n - 1} (k + 1) \prod_{i = 1}^{k}
%  \frac{n - i}{n - i + 2r - 1}}_{1 - \Prob{D_n = 0} = 1 - \alpha_n}
%  \right) \\ \\
%  &= \frac{n}{n + 2r + 1} \left( \frac{n - 1}{1 + r} + 2\right) \\ \\
%  &= \frac{n}{1 + r}
%\end{align*}
%%]]]

\end{document}